\documentclass{amsart}
\usepackage{fourier, bbm, color, enumitem}
\usepackage[foot]{amsaddr}

\setenumerate[1]{label=(\roman{enumi})}

\newcommand{\indicator}[1]{\mathbbm{1}_{\{#1\}}}
\newcommand{\norm}[1]{\left\lvert #1 \right\rvert}
\newcommand{\mv}[1]{{\bf #1}}

\theoremstyle{plain}
\newtheorem{theorem}{Theorem}[section]
\newtheorem{prop}[theorem]{Proposition}
\newtheorem{lemma}[theorem]{Lemma}
\newtheorem*{assumption*}{Assumption}
\newtheorem*{conjecture*}{Conjecture}

\numberwithin{equation}{section}

\theoremstyle{remark}

\def\E{\mathbb{E}}
\def\N{\mathbb{N}}
\def\P{\mathbb{P}}
\def\R{\mathbb{R}}

\def\Bcal{\Theta}

\def\Dcal{\mathcal{D}}

\def\Gcal{\mathcal{G}}
\def\Mcal{\mathcal{M}}

\def\d{\mathrm{d}}

\title{State space collapse for critical multistage epidemics}

\author{Florian Simatos}
\email{f.simatos@tue.nl}
\address{Department of Mathematics \& Computer Science, Eindhoven University of Technology, P.O.\ Box 513, 5600 MB Eindhoven, The Netherlands}

\date{\today}

\begin{document}

\begin{abstract}
	We study a multistage epidemic model which generalizes the SIR model and where infected individuals go through $K \geq 1$ stages of the epidemic before being removed. An infected individual in stage $k \in \{1, \ldots, K\}$ may infect a susceptible individual, who directly goes to stage $k$ of the epidemic; or it may go to the next stage $k{+}1$ of the epidemic. For this model, we identify the critical regime in which we establish diffusion approximations. Surprisingly, the limiting diffusion exhibits an unusual form of state space collapse which we analyze in detail.
\end{abstract}

\maketitle

\setcounter{tocdepth}{1}

\bigskip
\hrule
\vspace{-2mm}
\tableofcontents
\vspace{-8mm}
\hrule

\section{Introduction}

\subsection*{A multistage model}

In this paper we study an epidemic model proposed by Antal and Krapivsky~\cite{Antal12:0} which generalizes the classical SIR model. Similarly as for the SIR model, there is a closed population where each individual is either \emph{susceptible}, \emph{infected} or \emph{removed}. In addition, we assume that the disease which is spread progresses through $K \geq 1$ stages. For short, we will say that an (infected) individual is in stage $k \in \{1, \ldots, K\}$ if it possesses the disease in stage $k$. It will also be convenient to say that an individual is in stage $0$ if it is susceptible and in stage $K{+}1$ if it is removed.

Two types of transitions may occur within our model: either $(1)$ an infected individual in stage $k \in \{1, \ldots, K\}$ tries to infect a susceptible individual; or $(2)$ an infected individual in stage $k \in \{1, \ldots, K\}$ progresses to the next stage $k{+}1$. We consider the mean-field regime where in case $(1)$ above, the infected individual samples an individual uniformly at random from the whole population: if the sampled individual is susceptible, then it becomes infected and starts the infection at the same stage than the individual who infected him or her.

In view of its epidemiological interpretation and because it generalizes the SIR model (which corresponds to the single-stage case $K{=}1$), it is natural to interpret this model as an epidemic model. From that viewpoint, our model differs from multistage epidemic models previously proposed. For HIV/AIDS for instance, Hyman et al.~\cite{Hyman99:0} proposed a multistage (or staged-progression) model where a newly infected individual starts the epidemic in stage one. We believe that the techniques used in the present paper can be adapted to study this case as well.

Besides its natural epidemiological interpretation, this model may also bear insight into other areas of applied probability. We may for instance think of cell population dynamic, e.g., when a cancer tumor progresses, cells accumulate mutations that are passed on to newly infected cells. In this case the stage of an individual corresponds to the number of mutations. Another potential application of this model is in communication networks: in this case we may think of users accumulating information and, upon meeting with a user with no information, sharing all the information the first user has. Such kind of epidemic spreading of information is actually at the heart of modern peer-to-peer systems.

\subsection*{Scaling limits}

Epidemic models can broadly be categorized into discrete and continuous ones, see for instance Kendall~\cite{Kendall56:0}. Discrete models are individual-based models that focus on the inherent stochasticity of the dynamic, and are often described by finite state space Markov processes. Continuous models correspond to large population approximations: they may be either deterministic (corresponding to some notion of averaging in the large population regime) or stochastic. In the first case, the continuous model is typically described by ordinary differential equations (ODE) and in the second case, by stochastic differential equations (SDE).

In the present paper, we bridge these two viewpoints by starting from a discrete model which, after a suitable renormalization procedure, leads to a continuous model. This is a classical approach which allows to understand the extent to which a continuous model, usually more tractable, suitably approximates a discrete one. In addition to giving insight into how the epidemic develops over time, such schemes are also useful to understand the final outbreak size of the epidemic. In the SIR case for instance, this makes it possible to express the final outbreak size, properly renormalized, as the hitting time of $0$ by a Brownian motion with a parabolic drift; this allows in turn for explicit computation, see for instance Martin-L\"of~\cite{Martin-Lof98:0}. Note that our forthcoming Theorem~\ref{thm:outbreak} generalizes this result to the multistage case $K \geq 1$.

For the SIR model it is well-known that different scaling limits may be obtained depending on the discrete model's parameters, see for instance von Bahr and Martin-L\"of~\cite{Bahr80:0}. In particular, the limit may be deterministic or retain some stochasticity of the original discrete model, and the intermediate regime at which transition occurs is usually referred to as \emph{critical regime}: this is the regime in which we are interested in in the present paper. In this regime, scaling limits of the SIR model have for instance been obtained by Martin-L\"of et al.~\cite{Martin-Lof86:0, Martin-Lof98:0, Bahr80:0} and by Aldous~\cite{Aldous97:0} in the closely related context of the Erd\"os-R\'enyi random graph.

A popular method to establish such scaling limits relies on semi-group techniques, see for instance Ethier and Kurtz~\cite{Ethier86:0}. This approach has for instance recently been followed by Dolgoarshinnykh and Lalley~\cite{Dolgoarshinnykh06:0} to study the related SIS model, who also sketch a proof for the SIR model. However, because of the different boundary conditions between the SIS and the SIR models, the authors only mention that different analytical results from the theory of parabolic partial differential equations are needed in the SIR case. In the present paper we use a different and more probabilistic approach, closer to Aldous' approach~\cite{Aldous97:0}, that relies on semimartingale arguments. These arguments rely on elementary calculations made on the infinitesimal generator of the Markov process, and require almost no analytical results. This is the reason why we believe that this method is robust and can be adapted to study other models such as the HIV/AIDS model of Hyman et al.~\cite{Hyman99:0}.

\subsection*{Presentation and discussion of the main results}

Let us present our results in the particular case where all individual transition rates are equal to one: in general these will be allowed to be \emph{close} to one. For $n \geq 1$ we consider the $K{+}2$-dimensional Markov process $\mv{a}_n(t) = (a_{n,k}(t), k = 0, \ldots, K+1)$ that represents a multistage epidemic process in a population of size $n$. It is defined by an initial state $\mv{a}_n(0) = (a_{n,k}(0), k = 0, \ldots, K+1)$ satisfying the relation $a_{n,0}(0) + \cdots + a_{n,K+1}(0) = n$ and, for $a = (a_0, a_1, \ldots, a_{K+1}) \in \N^{K+2}$, by the following transition rates:
\begin{equation} \label{eq:transition-rates-a_n-critical}
	a \longrightarrow \begin{cases}
		a - e_k + e_{k+1} & \text{ at rate } a_k, \ \ 1 \leq k \leq K,\\
		a - e_0 + e_k & \text{ at rate } a_k a_0/n, \ \ 1 \leq k \leq K,
	\end{cases}
\end{equation}
where $e_0 = (1,0,\ldots, 0)$, $e_1 = (0,1,0,\ldots, 0)$, etc. For $t \geq 0$, $a_{n,k}(t)$ is the number of individuals in stage $k$ at time $t$, i.e., $a_{n,0}(t)$ is the number of susceptible individuals; $a_{n,k}(t)$ for $k = 1, \ldots, K$ is the number of infected individuals in stage $k$; and $a_{n,K+1}(t)$ is the number of removed individuals. The transition rates~\eqref{eq:transition-rates-a_n-critical} preserve the total population size, so that the relation $a_{n,0}(t) + \cdots + a_{n,K+1}(t) = n$ is satisfied for any $t \geq 0$.

These rates formalize the description of the dynamic given above: each infected individual progresses to the next stage of the epidemic at rate one; each infected individual makes an infection attempt at rate one, which is successful with probability proportional to the number of susceptible individuals (which gives the factor $a_0/n$).
\\

Consider the initial condition where $a_{n,1}(0) = n^\beta$ for some $\beta \in (0,1)$ and $a_{n,k}(0) = 0$ for $k = 2, \ldots, K+1$. The asymptotic behavior of the $K{+}1$-dimensional process $(a_{n,k}, k = 1, \ldots, K+1)$ as $n \to +\infty$ is then governed by the precise value of $\beta$, and there are three different regimes:
\begin{description}
	\item[Small initial condition $\beta < 1/(K+2)$] on the time scale $n^\beta$, $a_{n,k}$ lives on the space scale $n^{k \beta}$ and the scaling limit is governed by the following SDE:
	\begin{equation} \label{eq:SL-small-critical}
		\begin{cases}
			\d A_1(t) = (2 A_1(t))^{1/2} \d B(t), & \\
			\d A_k(t) = A_{k-1}(t) \d t &  \ \text{ for } \ k = 2, \ldots, K+1;
		\end{cases}
	\end{equation}
	\item[Intermediate initial condition $\beta = 1/(K+2)$] on the time scale $n^{1/(K+2)}$, $a_{n,k}$ lives on the space scale $n^{k/(K+2)}$, the outbreak size is of the order of $n^{(K+1)/(K+2)}$ and the scaling limit is governed by the following SDE:
	\begin{equation} \label{eq:SL-intermediate-critical}
		\begin{cases}
			\d A_1(t) = - A_{K+1}(t) A_1(t) \d t + (2 A_1(t))^{1/2} \d B(t), & \\
			\d A_k(t) = (A_{k-1}(t) - A_{K+1}(t) A_k(t)) \d t &  \ \text{ for } \ k = 2, \ldots, K,\\
			\d A_{K+1}(t) = A_K(t) \d t; & 
		\end{cases}
	\end{equation}
	\item[Large initial condition $\beta > 1/(K+2)$] on the time scale $n^\gamma$ with $\gamma = (1{-}\beta)/(K{+1})$, $a_{n,k}$ lives on the space scale $n^{\beta + (k-1) \gamma}$ and the scaling limit is governed by the following ODE:
	\begin{equation} \label{eq:SL-large-critical}
		\begin{cases}
			\d A_1(t) = - A_{K+1}(t) A_1(t) \d t, & \\
			\d A_k(t) = (A_{k-1}(t) - A_{K+1}(t) A_k(t)) \d t &  \ \text{ for } \ k = 2, \ldots, K,\\
			\d A_{K+1}(t) = A_K(t) \d t. & 
		\end{cases}
	\end{equation}
\end{description}

In the above, $B$ is a standard Brownian motion. Moreover, since the total population size is fixed, the asymptotic behavior of $a_{n,0}$ is recovered from the behavior of the $a_{n,k}$'s: since $n - a_{n,0} = a_{n,1} + \cdots + a_{n,K+1}$ and since in all three regimes $a_{n,K+1} \gg a_{n,k}$ for $k = 1, \ldots, K$, $n - a_{n,0}$ obeys the same scaling and has the same limit than $a_{n,K+1}$.

It is also interesting to note that the intermediate regime interpolates between the small and large ones in two ways: $(1)$ the scaling at play there can be obtained by letting $\beta \uparrow 1/(K{+}2)$ in the small regime or $\beta \downarrow 1/(K{+}2)$ in the large regime, and $(2)$ the evolution of $A_1$ in~\eqref{eq:SL-intermediate-critical} is a mixture of~\eqref{eq:SL-small-critical} and~\eqref{eq:SL-large-critical}. However, $\beta$ does not appear in the asymptotic dynamics~\eqref{eq:SL-small-critical}--\eqref{eq:SL-large-critical} and so~\eqref{eq:SL-intermediate-critical} appears somehow discontinuously.
\\

When specializing the results of the intermediate regime to $K = 1$ we recover the classical result on the SIR model~\cite{Bahr80:0}: starting with of the order of $n^{1/3}$ infected individuals, we end up with an outbreak size of the order of $n^{2/3}$. When $K = 2$, these results show that starting with $n^{1/4}$ individuals in the first stage, the number of individuals in the second stage during the outbreak is of the order $n^{1/2}$, while the total outbreak size is of the order of $n^{3/4}$.

Note that for the above comparison we have referred to the work by von Bahr and Martin-L\"of~\cite{Bahr80:0} which actually considers a discrete-time version of the SIR model, the so-called Reed-Frost model. In this model, individuals stay infected for a deterministic amount of time and, at the end of this period, try to infect every other individual with a fixed probability. In the case $K = 1$ the two models are indeed equivalent, at least with respect to the final outbreak size: heuristically, it does not matter when and by whom a given individual gets infected. However, in the truly multistage case $K \geq 2$ this is no longer the case: the stage of an individual that spreads the disease matters. In that respect, individuals from different stages ``compete'' for the pool of susceptible individuals and it is natural to consider a continuous-time model. We suspect that, compared to the case $K=1$, this additional difficulty is causing the discrepancy discussed in Section~\ref{sec:conjecture}.
\\

In order to get insight into the scalings at play, let us consider the continuous-time Markov branching process $\mv{a}^B = (a^B_k, k = 1, \ldots, K+1)$ with $K+1$ types, whose transition rates for $a \in \N^{K+1}$ are given by
\begin{equation} \label{eq:transition-rates-BP-critical}
	a \longrightarrow \begin{cases}
		a - e_k + e_{k+1} & \text{ at rate } a_k, \ \ 1 \leq k \leq K,\\
		a + e_k & \text{ at rate } a_k, \ \ 1 \leq k \leq K.
	\end{cases}
\end{equation}

These rates are obtained from~\eqref{eq:transition-rates-a_n-critical} by setting $a_0 = n$: this corresponds to an infinite population limit, where an infection attempt is always successful. Then $a^B_1$ is a continuous-time Markov branching process (also known as Bellman--Harris branching process), and it can be seen that if $a^B_1(0) = n^\beta$ and time is sped up by $n^\beta$, then it converges to an interesting object (Feller diffusion). Thus if we are to have an interesting limit for $a^B_1$, the correct time scale (for all processes) needs to be $n^\beta$.

But then, the total number individuals it ever begets is of the order of $\int_0^\infty a^B_1(t) \d t$ which is of the order of $n^{2 \beta}$. These individuals form the initial condition for $a^B_2$, which can thus be seen as a Bellman--Harris branching process, started at $n^{2\beta}$ but for which time is sped up by $n^\beta$. However, we know from the above discussion that in order for $a^B_2$ to evolve on the space scale $n^{2 \beta}$, it needs to be sped up in time by $n^{2 \beta}$: thus on the time scale $n^\beta$, $a^B_2$ remains essentially constant. However, on this time scale it begets of the order of $\int_0^{n^\beta} a^B_2(t) \d t \propto n^{3 \beta}$ individuals, which form the basis for $a^B_3$. Iterating this argument, we see the geometric progression observed in all three regimes appearing, where $a_{n,k+1}$ is of the order of the time scale times $a_{n,k}$. But actually, more can be learned from this simple approximation.

First of all, we see that the approximation of $\mv{a}_n$ by $\mv{a}^B$ is asymptotically exact in the small regime, in the sense that~\eqref{eq:SL-small-critical} is also the scaling limit of $\mv{a}^B$ subject to the same scaling. This approximation begins to break down in the intermediate regime, which is therefore the regime where \emph{finite-size population effects} begin to kick in. This phenomenon is also sometimes called the \emph{depletion of points effect}: in the intermediate regime the epidemic begins to feel the decreasing number of susceptible individuals.

Moreover, we have just argued that we are considering a time scale which is suited for $a_{n,1}$ but not for $a_{n,2}$. More precisely, on the time scale where $a_{n,1}$ evolves, $a_{n,2}$ has not enough time to evolve on its own and all its randomness comes from $a_{n,1}$. This problem of time scales leads to a classical phenomenon in queueing theory, called \emph{state space collapse}, but which is more surprising to find in the context of epidemic.

\subsection*{An unusual form of state space collapse}

Let $\bar K = K-1$ and $T_n$ be the first time at which $a_{n,1}$ hits zero, i.e., $T_n = \inf\{ t \geq 0: a_{n,1}(t) = 0 \}$. It is then important to observe that the $\bar K {+} 2$-dimensional process $\bar{\mv{a}}_n = (\bar{a}_{n,k}, k = 0, \ldots, \bar K + 1)$ obtained by keeping track of the individuals in stages $0, 2, \ldots, K+1$ after time $T_n$, i.e., with $\bar{a}_{n,0}(t) = a_{n,0}(T_n + t)$ and $\bar{a}_{n,k}(t) = a_{n,k+1}(T_n + t)$ for $k = 1, \ldots, \bar K+1$, is a multistage epidemic process with $\bar K$ stages. However, we know that in the intermediate regime, $a_{n,2}$, and thus $\bar a_{n,1}(0)$, is of the order of $n^{2/(K+2)} = n^{2/(\bar K + 3)}$. As such, this corresponds to a multistage epidemic process with a large initial condition (with $\beta = 2/(\bar K {+} 3) > 1/(\bar K {+} 2)$), and we know from the above discussion that its scaling limit is deterministic. In summary, once the first stage has got extinct, the subsequent evolution of the process is deterministic, which supports the previous claim that all the randomness lies in the first stage. This property can actually be directly read off the SDE~\eqref{eq:SL-intermediate-critical}.

Indeed, a striking feature of this SDE is that the diffusion coefficients governing $A_k$ for $k \neq 1$ are equal to zero: in particular, only the coordinate $A_1$ is truly random, the other ones being obtained deterministically from it via an ODE. More precisely, let $F: \R \times \R^K \to \R^K$ be the function such that, if $\bar{\mv{A}} = (A_2, \ldots, A_{K+1})$, then~\eqref{eq:SL-intermediate-critical} rewrites $\d \bar{\mv{A}}(t) = F(A_1(t), \bar{\mv{A}}(t)) \d t$: then $\bar{\mv{A}}$ satisfies an ODE of the kind
\begin{equation} \label{eq:ODE-critical}
	\frac{\d x(t)}{\d t} = F\left( y(t), x(t) \right), \ t \geq 0.
\end{equation}

It will be seen that if $y: [0,\infty) \to [0,\infty)$ is continuous, then this ODE has a unique solution. This naturally defines a one-dimensional manifold $\Mcal(y) \subset \R^K$, namely, if $x_y$ is the unique solution to~\eqref{eq:ODE-critical}, the space $\Mcal(y) = \left\{ x_y(t): t \geq 0 \right\}$.

If $\mv{A} = (A_1, \ldots, A_{K+1})$ is the solution to~\eqref{eq:SL-intermediate-critical}, then by definition $\bar{\mv{A}}$ is the unique solution to~\eqref{eq:ODE-critical} with $y = A_1$, which justifies seeing $\bar{\mv A}$ as deterministically obtained from $A_1$. Thus we can see $A_{K+1}$ as a function of $A_1$ and the initial condition $\bar{a} = \bar{\mv{A}}(0)$, and write $A_{K+1}(t) = \varphi(A_1, \bar a; t)$ for some deterministic map $\varphi$. With this notation, we see that $A_1$ on its own satisfies the following SDE
\begin{equation} \label{eq:SDE-A_1-critical}
	\d A_1(t) = \left(\gamma_1 - \varphi(A_1, \bar a ; t) \right) A_1(t) \d t + (2A_1(t))^{1/2} \d B(t).
\end{equation}

Such an SDE is sometimes referred to as an It\^o process and is, in the terminology of Rogers and Williams~\cite{Rogers87:0}, not of diffusion type since $\varphi(A_1, \bar a ; t)$ depends on $A_1$ through the entire path $(A_1(s), 0 \leq s \leq t)$ and not only through the value of $A_1$ at time $t$.
\\

In summary, we start with a sequence of $K{+}2$-dimensional stochastic processes (or $K{+}1$-dimensional, given that the total population size is fixed), but the diffusion $\mv{A} = (A_1, \ldots, A_{K+1})$ appearing in the limit is truly one-dimensional: one coordinate, $A_1$, is random and given by~\eqref{eq:SDE-A_1-critical}, while the other coordinates $A_2, \ldots, A_{K+1}$ are obtained deterministically from it via~\eqref{eq:ODE-critical}.

Such a phenomenon of reduction of the dimension of the state space in the critical (or near-critical) regime is well-known in queueing theory, where it is usually referred to as \emph{state space collapse}. It has first been systematically investigated in~\cite{Bramson98:0, Williams98:0} in the context of multiclass queueing networks, and has since then been observed in various settings~\cite{Borst13:0, Reiman84:0, Stolyar04:0, Verloop11:0}. In all these examples and similarly as what we observe in our model, the randomness of the limiting diffusion is contained in one coordinate (the workload process for queueing systems), while the other coordinates are obtained deterministically from it. In queueing theory typically, the workload process converges to a reflected Brownian motion $W$ and the queue length processes $\mv{Q}$ are obtained via a deterministic map, e.g., $\mv{Q} = G \circ W$.

However, there are two main differences between these results from queueing theory and our present form of state space collapse. First of all, in all these queueing examples, the map $G$ is linear, i.e., we can write $Q_k(t) = \alpha_k W(t)$ for some deterministic coefficients $\alpha_k > 0$. This makes the queue length processes $\mv{Q} = (Q_k)$ live in a \emph{deterministic} one-dimensional manifold, namely, the space $\Mcal_Q = \{ (\alpha_k w): w \geq 0 \}$. In contrast, the coordinates $A_2, \ldots, A_{K+1}$ live in our case in the \emph{random} one-dimensional manifold $\Mcal(A_1)$.

Moreover, the reasons leading to this state space collapse phenomenon are also distinct. As explained above, in our case this comes from a matter of time scales while in queueing theory, this essentially comes from the law of large numbers since the workload is the sum of the residual service times over the number of customers.

\subsection*{Organization of the paper}

In the next section we introduce the general model considered in the present paper, which generalizes~\eqref{eq:transition-rates-a_n-critical} to the near-critical case, and state the main results, which formalize the results described previously. In Section~\ref{sec:analysis-SDE} we analyze the limiting SDE that generalizes~\eqref{eq:SL-intermediate-critical}, where we prove existence and uniqueness of solutions as well as a useful sample path property. These results rely on the analysis of the ODE that generalizes~\eqref{eq:ODE-critical}, and the deterministic results that are needed are proved in the Appendix~\ref{appendix:proof-ODE}. Sections~\ref{sec:proof-scaling-limit} and~\ref{sec:proof-outbreak} are devoted to the proofs of the main results: Section~\ref{sec:proof-scaling-limit} gives the proof for the scaling limits, i.e., the asymptotic behavior at the process level, and Section~\ref{sec:proof-outbreak} is concerned with the proof of the asymptotic behavior of the outbreak size. Finally, we conclude the paper by discussing in Section~\ref{sec:conjecture} the conjecture formulated in~\cite{Antal12:0}, and how it relates to our results.

\subsection*{Acknowledgements} I am extremely grateful to Remco van der Hofstad and Johan van Leeuwaarden for introducing me to the topic of this paper and for very stimulating discussions on this particular model. I would also like to thank the two anonymous reviewers whose comments greatly helped to improve the exposition of the paper.

\section{Main results} \label{sec:main-results}

\subsection{Notation} \label{sub:notation}

Let in the sequel $\Dcal$ be the set of real-valued, c\`adl\`ag functions. For $f \in \Dcal$ and $\varepsilon > 0$, we define $T_0(f) = \inf\{ t \geq 0: f(t) = 0 \}$, $T^\uparrow_\varepsilon(f) = \inf\{ t \geq 0: f(t) \geq \varepsilon \}$ and $T^\downarrow_\varepsilon(f) = \inf\{ t \geq 0: f(t) \leq \varepsilon \}$, with the convention $\inf \emptyset = +\infty$. If $t \geq 0$ and $f$ is c\`adl\`ag (real- or vector-valued) we consider $\theta_t$ the shift operator, defined by $\theta_t(f) = f(t \, + \cdot \,)$. The space of c\`adl\`ag functions with values in $\R^d$ is endowed with the topology of uniform convergence on compact sets.

In the sequel we fix some integer $K \geq 1$ and we consider the map $\bar \pi: \R^{K+2} \to \R^{K+1}$ defined by $\bar \pi(a_0, \ldots, a_{K+1}) = (a_0, a_2, \ldots, a_{K+1})$ (this unusual indexing of vectors will be convenient for our purposes).

\subsection{Model and main results}

We now present the full model investigated in the rest of the paper, which generalizes~\eqref{eq:transition-rates-a_n-critical} by allowing individual transition rates to be close to one. For each $n \geq 1$, let $\mv{a}_n(t) = (a_{n,k}(t), k = 0, \ldots, K+1)$ be the $K{+}2$-dimensional Markov process corresponding to a finite population of size $n$, i.e., $a_{n,0}(0) + \cdots + a_{n,K+1}(0) = n$, and with non-zero transition rates given for $a = (a_0, a_1, \ldots, a_{K+1}) \in \N^{K+2}$ by
\begin{equation} \label{eq:transition-rates-a_n}
	a \longrightarrow \begin{cases}
		a - e_k + e_{k+1} & \text{ at rate } (1+\delta_{n,k}) a_k, \ \ 1 \leq k \leq K,\\
		a - e_0 + e_k & \text{ at rate } (1+\varepsilon_{n,k}) a_k a_0/n, \ \ 1 \leq k \leq K,
	\end{cases}
\end{equation}
where $\delta_{n,k}, \varepsilon_{n,k} > -1$ for $k = 1, \ldots, K$, and we will always assume that $\delta_{n,k}, \varepsilon_{n,k} \to 0$ as $n \to +\infty$ for each $k = 1, \ldots, K$: the critical case~\eqref{eq:transition-rates-a_n-critical} is recovered by setting $\delta_{n,k} = \varepsilon_{n,k} = 0$. Moreover, it will be convenient to define $\delta_{n,0} = \delta_{n,K+1} = \varepsilon_{n,K+1} = -1$. We first state the results concerning the intermediate regime, which constitute the main results of the paper.

\begin{theorem} [Intermediate initial condition] \label{thm:SL-intermediate}
	Let $\beta = 1/(K{+}2)$ and assume that for each $k = 1, \ldots, K$, there exists $\gamma_k \in \R$ such that
	\begin{equation} \label{eq:gamma-intermediate}
		n^\beta \left( \varepsilon_{n,k} - \delta_{n,k} \right) \mathop{\longrightarrow}_{n \to +\infty} \gamma_k.
	\end{equation}
	
	Let $\mv{A}_n = (A_{n,k}, k = 0, \ldots, K+1)$ be the $K{+}2$-dimensional process defined as follows:
	\begin{equation} \label{eq:scaling-intermediate}
		A_{n,0}(t) = \frac{n - a_{n,0}(n^\beta t)}{n^{(K+1)\beta}} \ \text{ and } \ A_{n,k}(t) = \frac{a_{n,k}(n^\beta t)}{n^{k\beta}}, \ k = 1, \ldots, K+1, \ t \geq 0.
	\end{equation}
	
	If $\mv{A}_n(0) \to a \in [0,\infty)^{K+2}$, then the sequence of processes $(\mv{A}_n, n \geq 1)$ converges weakly as $n \to +\infty$ to the unique solution $\mv{A} = (A_k, 0 \leq k \leq K+1)$ of the SDE
	\begin{equation} \label{eq:SL-intermediate}
		\mv{A}(t) = a + \int_0^t b(\mv A(u)) \d u + \int_0^t \sigma(\mv A(u)) \d B(u),
	\end{equation}
	where $B$ is a standard $K{+}2$-dimensional Brownian motion and for $a = (a_0, \ldots, a_{K+1}) \in [0,\infty)^{K+2}$ we have $\sigma(a) = \text{\textnormal{diag}}(0, (2 a_1)^{1/2}, 0, \ldots, 0)$ and
	\[ \left\{\begin{array}{l}
		b_0(a) = b_{K+1}(a) = a_K,\\
		b_1(a) = (\gamma_1 - a_0) a_1,\\
		b_k(a) = a_{k-1} + (\gamma_k - a_0) a_k, \ \ k = 2, \ldots, K.
	\end{array}\right. \]
\end{theorem}

Since all the diffusion coefficients but one are equal to $0$, when convenient we will identify the $(K{+}2) \times (K{+}2)$ matrix $\sigma(a)$ and its only non-zero entry $(2 a_1)^{1/2}$. For the same reason, if we write $B = (B_0, \ldots, B_{K+1})$ then only the coordinate $B_1$ matters and we will identify $B$ and $B_1$, i.e., whenever convenient we will consider that $B$ is a one-dimensional standard Brownian motion; note that this is in line with the notation used in~\eqref{eq:SL-small-critical} and~\eqref{eq:SL-intermediate-critical}. We now state some properties of solutions to~\eqref{eq:SL-intermediate}.

\begin{prop}\label{prop:A}
	Consider the assumptions and notation of Theorem~\ref{thm:SL-intermediate}. Then $\mv{A}$ almost surely satisfies the following properties:
	\begin{enumerate}
		\item the two processes $A_0$ and $A_{K+1}$ are equal;
		\item if $K \geq 2$, then $A_0$ is strictly increasing, and its terminal value $A_0(\infty)$ is finite and satisfies $A_0(\infty) > \gamma_k$ for every $k = 2, \ldots, K$;
		\item $T_0(A_1) < +\infty$;
		\item for every $k = 2, \ldots, K$, $A_k(t) \to 0$ as $t \to +\infty$ but $T_0(A_k) = +\infty$.
	\end{enumerate}
\end{prop}

Note that the first property of $\mv{A}$ above comes from the relation $\sum_{k=0}^{K+1} a_{n,k}(t) = n$, which translates after scaling to $\sum_{k=1}^{K+1} n^{k \beta} A_{n,k}(t) = n^{(K+1) \beta} A_{n,0}(t)$. Next, we turn to the asymptotic behavior of the outbreak size $A_{n, K+1}(\infty)$: the following result essentially states that we can interchange the limits $n \to +\infty$ and $t \to +\infty$.

\begin{theorem}\label{thm:outbreak}
	Under the assumptions and notation of Theorem~\ref{thm:SL-intermediate}, the renormalized outbreak size $A_{n,K+1}(\infty)$ converges weakly as $n \to +\infty$ toward $A_{K+1}(\infty)$.
\end{theorem}

Let us conclude the presentation of results in the intermediate regime by comparing with the SIR case $K = 1$. In the case $K = 1$ the natural notion of criticality is that every infected individual tries to infect in average (close to) one other individual. The notion of criticality that we consider here is at first sight different, since every infected individual tries to infect in average one other individual \emph{in each stage}. If an individual starts in stage $k$, it will therefore make in average $K{-}k{+}1$ infection attempts. But the two notions actually coincide. Indeed, Theorem~\ref{thm:SL-intermediate} shows that individuals in the last stage $K$ of the epidemic dominate: there are $n^{K/(K+2)}$ such individuals and $n^{k/(K+2)} \ll n^{K/(K+2)}$ individuals at stage $k = 1, \ldots, K-1$. Thus with overwhelming probability, an individual eventually infected has actually started the epidemic in stage $K$ and has made in average one infection attempt. Beware however: this does not mean that only stage $K$ matters, since $a_{n,K}$ becomes of the order of $n^{K/(K+2)}$ thanks to the help of the individuals in the previous stages!
	
Moreover, according to Proposition~\ref{prop:A}, the case $K = 1$ is the only case where $A_{K+1}(\infty)$ is reached in finite time. Indeed, for $K = 1$ the dynamic is frozen after the time $T_0(A_1)$ which is finite, while if $K \geq 2$ then $A_{K+1}$ remains strictly increasing at all times. As will be discussed in Section~\ref{sec:proof-outbreak} when proving Theorem~\ref{thm:outbreak}, this difference creates an additional difficulty in order to control the asymptotic behavior of the outbreak size.
\\

We now state the result for the large regime.

\begin{prop}[Large initial condition] \label{prop:SL-large}
	Fix some sequence $n^{1/(K+2)} \ll \alpha_{n,1} \ll n$ and for $n \geq 1$ define
	\begin{equation} \label{eq:scaling-constants-large}
\tau_n = \left( \frac{n}{\alpha_{n,1}} \right)^{1/(K+1)}, \ \alpha_{n,0} = \frac{n}{\tau_n} \ \text{ and } \ \alpha_{n,k} = \tau_n^{k-1} \alpha_{n,1} \ \text{ for } \ k = 2, \ldots, K+1.
	\end{equation}
	
	Assume that for each $k = 1, \ldots, K$, there exists $\gamma_k \in \R$ such that
	\begin{equation} \label{eq:gamma-large}
		\tau_n \left( \varepsilon_{n,k} - \delta_{n,k} \right) \mathop{\longrightarrow}_{n \to +\infty} \gamma_k.
	\end{equation}
	
	Let $\mv{A}_n = (A_{n,k}, k = 0, \ldots, K+1)$ be the $K{+}2$-dimensional renormalized process defined as follows:
	\begin{equation} \label{eq:scaling-large}
		A_{n,0}(t) = \frac{n - a_{n,0}(\tau_n t)}{\alpha_{n,0}} \ \text{ and } \ A_{n,k}(t) = \frac{a_{n,k}(\tau_n t)}{\alpha_{n,k}}, \ k = 1, \ldots, K+1, \ t \geq 0.
	\end{equation}
	
	If $\mv{A}_n(0) \to a \in [0,\infty)^{K+2}$, then the sequence of processes $(\mv{A}_n, n \geq 1)$ converges weakly as $n \to +\infty$ to the unique solution of the ODE
	\begin{equation} \label{eq:SL-large}
		\mv{A}(t) = a + \int_0^t b(\mv{A}(u)) \d u,
	\end{equation}
	with $b$ defined in Theorem~\ref{thm:SL-intermediate}.
\end{prop}

Note that this result is coherent with Theorem~\ref{thm:SL-intermediate}, in the sense that the limit of $\theta_{T_0(A_{n,1})} \circ \bar \pi \circ \mv{A}_n = \bar \pi \circ \theta_{T_0(A_{n,1})} \circ \mv{A}_n$ is the same whether we consider this process as the process $\bar \pi \circ \mv{A}_n$ shifted at time $T_0(A_{n,1})$ and then use Theorem~\ref{thm:SL-intermediate}, or whether we consider this process as a multistage epidemic process with $\bar K = K-1$ stages started from a large initial condition and then use Proposition~\ref{prop:SL-large}.
\\

We complete the results of Theorem~\ref{thm:SL-intermediate} and Proposition~\ref{prop:SL-large} by studying the case of a small initial condition $a_{n,1}(0) \ll n^{1/(K+2)}$. When $a_{n,1}(0)$ converges to some finite number, the sequence of processes $(a_{n,k}, k = 1, \ldots, K+1)$ converges to the multitype branching process given by~\eqref{eq:transition-rates-BP-critical}. As the next result shows, such a branching approximation continues to be valid in the small regime.

\begin{prop} [Small initial condition] \label{prop:SL-small}
	Fix some sequence $1 \ll \alpha_{n,1} \ll n^{1/(K+2)}$ and assume that for each $k = 1, \ldots, K$, there exists $\gamma_k \in \R$ such that
	\begin{equation} \label{eq:gamma-small}
		\alpha_{n,1} \left( \varepsilon_{n,k} - \delta_{n,k} \right) \mathop{\longrightarrow}_{n \to +\infty} \gamma_k.
	\end{equation}
	
	Let $\mv{A}_n = (A_{n,k}, k = 0, \ldots, K+1)$ for $n \geq 1$ be the $K{+}2$-dimensional renormalized process defined as follows:
	\begin{equation} \label{eq:scaling-small}
		A_{n,0}(t) = \frac{n - a_{n,0}(\alpha_{n,1} t)}{\alpha_{n,1}^{K+1}} \ \text{ and } \ A_{n,k}(t) = \frac{a_{n,k}(\alpha_{n,1} t)}{\alpha_{n,1}^k}, \ k = 1, \ldots, K+1, \ t \geq 0.
	\end{equation}
	
	If $\mv{A}_n(0) \to a \in [0,\infty)^{K+2}$, then the sequence of processes $(\mv{A}_n, n \geq 1)$ converges weakly as $n \to +\infty$ to the unique solution of the SDE
	\begin{equation} \label{eq:SL-small}
		\mv{A}(t) = a + \int_0^t b^S(\mv{A}(u)) \d u + \int_0^t \sigma(\mv{A}(u)) \d B(u),
	\end{equation}
	where $\sigma$ is as in Theorem~\ref{thm:SL-intermediate} and $b^S$ is given by
	\[ \left\{\begin{array}{l}
		\smallskip b^S_0(a) = b^S_{K+1}(a) = a_K,\\
		\smallskip b^S_1(a) = \gamma_1 a_1,\\
		b^S_k(a) = a_{k-1} + \gamma_k a_k, \ \ k = 2, \ldots, K.
	\end{array}\right. \]
\end{prop}

As discussed in the introduction for the strictly critical case $\delta_{n,k} = \varepsilon_{n,k} = 0$, the limiting diffusion~\eqref{eq:SL-small} obtained in the small regime is the same as the limit of the branching process corresponding to the infinite population setting, i.e., where $a_0 / n = 1$ in~\eqref{eq:transition-rates-a_n} (this could be proved using the techniques of the present paper). Thus $a_{n,1}(0) \propto n^{1/(K+2)}$ is the threshold at which finite-size population effects (or depletion of points effects) begin to kick in: this is the threshold at which the branching approximation ceases to be valid.

Note also that for $\mv{A} = (A_k, k = 0, \ldots, K+1)$ satisfying~\eqref{eq:SL-small}, $A_1$ is Feller diffusion with drift $\gamma_1$, i.e., is the unique solution to the SDE $\d A_1 = \gamma_1 A_1 \d t + (2 A_1)^{1/2} \d B$ (see~\eqref{eq:SDE-Feller} below), and then $A_k$ for $k = 2, \ldots, K+1$ is given recursively by $A'_k = \gamma_k A_k + A_{k-1}$ (here and in the sequel, prime denotes differentiation with respect to the time variable). In particular, existence and uniqueness of solutions to~\eqref{eq:SL-small} follow immediately.
\\

The next three sections are devoted to proving these results. In the next section we prove uniqueness and existence to solutions of~\eqref{eq:SL-intermediate}, and we also prove Proposition~\ref{prop:A}. The scaling limits results of Theorem~\ref{thm:SL-intermediate} and Propositions~\ref{prop:SL-large} and~\ref{prop:SL-small} are proved in Section~\ref{sec:proof-scaling-limit}, and the proof of Theorem~\ref{thm:outbreak} on the asymptotic behavior of the outbreak size is given in Section~\ref{sec:proof-outbreak}.

\section{Analysis of the SDE~\eqref{eq:SL-intermediate}} \label{sec:analysis-SDE}

As mentioned in the introduction, the fact that $\sigma(a) = \text{\textnormal{diag}}(0, (2 a_1)^{1/2}, 0, \ldots, 0)$ is a manifestation of the state space collapse property, and it makes the process $\bar \pi \circ \mv{A}$ deterministically obtained from $A_1$ by an ODE. More precisely, we consider in the sequel $F: \R \times \R^{K+1} \to \R^{K+1}$ the function defined by $F(a_1,\bar \pi(a)) = \bar \pi(b(a))$ for $a \in \R^{K+2}$ and with $b$ as in Theorem~\ref{thm:SL-intermediate} (when all parameters $\gamma_k = 0$, this coincides with the function $F$ in~\eqref{eq:ODE-critical}). With this notation, we see that if $\mv{A}$ satisfies~\eqref{eq:SL-intermediate}, then $\bar \pi \circ \mv{A}$ is a solution of the ODE
\begin{equation} \label{eq:ODE}
	\frac{\d x(t)}{\d t} = F(y(t), x(t)), \ x(0) = \bar a,
\end{equation}
with $y = A_1$. The following properties of this ODE will be needed: their proof is postponed to the Appendix~\ref{appendix:proof-ODE}.

\begin{lemma} \label{lemma:stab-ODE}
	For any $\bar a \in [0,\infty)^{K+1}$ and any continuous function $y: [0,\infty) \to [0,\infty)$, the ODE~\eqref{eq:ODE} has a unique solution defined on $[0,\infty)$.
\end{lemma}

\begin{lemma} \label{lemma:properties-ODE}
	Fix $\bar a = (a_0, a_2, \ldots, a_{K+1}) \in [0,\infty)^{K+1}$ and $y: [0,\infty) \to [0,\infty)$ a continuous function, and let $x = (x_0, x_2, \ldots, x_{K+1})$ be the unique solution to~\eqref{eq:ODE} given by Lemma~\ref{lemma:stab-ODE}. Then the function $x_0$ is non-decreasing and its limit $x_0(\infty)$ as $t \to +\infty$ exists in $[0,\infty]$. Moreover:
	\begin{enumerate}
		\item \label{ODE:>0} if $y(0) > 0$ or $a_2 > 0$, then $x_k(t) > 0$ for every $k = 0, 2, \ldots, K$ and $t > 0$;
		\item \label{ODE:boundedness} if $x_k(t) > 0$ for every $k = 2, \ldots, K$ and $t > 0$ and $x_0$ is bounded, then $\int_0^\infty y < +\infty$ and moreover $\int_0^\infty x_k < +\infty$ and $x_0(\infty) > \gamma_k$ for every $k = 2, \ldots, K$;
		\item \label{ODE:to0} if $y(t) = 0$ for all $t \geq 0$ and $a_2 > 0$, then $x_0$ is bounded and $x_k(t) \to 0$ as $t \to +\infty$ for every $k = 2, \ldots, K$.
	\end{enumerate}
\end{lemma}

We now turn to the study of the SDE~\eqref{eq:SL-intermediate}. In the sequel we will call Feller diffusion with drift $\gamma \in \R$ the unique solution to the SDE
\begin{equation} \label{eq:SDE-Feller}
	Z(t) = Z(0) + \gamma \int_0^t Z(u) \d u + \int_0^t (2Z(u))^{1/2} \d B(u).
\end{equation}

It is well-known that there is a unique strong solution to~\eqref{eq:SDE-Feller}. If $Z$ is this solution, then it does not explode in finite time, $\P(T_0(Z) < +\infty) > 0$ and $\P(T_0(Z) < +\infty) = 1$ if $\gamma \leq 0$. Moreover, since $A_1$ satisfies
\[ A_1(t) = A_1(0) + \int_0^t b_1(\mv{A}(u)) \d u + \int_0^t (2A_1(u))^{1/2} \d B(u) \]
with $b_1(a) = (\gamma_1 - a_0) a_1 \leq \gamma_1 a_1$, Theorem~V.$43.1$ in Rogers and Williams~\cite{Rogers87:0} implies that $A_1(t) \leq Z(t)$ for all $t \geq 0$ almost surely, where $Z$ is Feller diffusion with drift $\gamma_1$ started at $Z(0) \geq A_1(0)$. This comparison argument will be used several times. The proof of the following result uses standard arguments, and we only sketch the proof.

\begin{lemma} \label{lemma:SDE-exact}
	Uniqueness in law holds for the SDE~\eqref{eq:SL-intermediate}.
\end{lemma}

\begin{proof}
	The problem to invoke classical results is that some of the coefficients $b_k$ grow quadratically. However, standard localization and change of drift arguments (to go back to the case of Feller diffusion, since $\sigma(a) = \text{\textnormal{diag}}(0, (2 a_1)^{1/2}, 0, \ldots, 0)$) show that for each $N \geq 1$, the law of a solution $\mv{A}$ to~\eqref{eq:SL-intermediate} stopped at $\inf\{t \geq 0: \left\lVert \mv{A}(t) \right\rVert \geq N\}$ is uniquely determined (say, with $\lVert a \rVert = a_0 + \cdots + a_{K+1}$). By successively patching up these solutions, we obtain uniqueness to~\eqref{eq:SL-intermediate} until the time of explosion, and so it only remains to show that solutions to~\eqref{eq:SL-intermediate} do not explode. But $A_1$ cannot explode since it is dominated by Feller diffusion, and since $\bar \pi \circ \mv{A}$ satisfies~\eqref{eq:ODE}, no other coordinate can explode by Lemma~\ref{lemma:stab-ODE}.
\end{proof}

\begin{lemma} \label{lemma:properties-SDE}
	If $\mv{A}$ satisfies~\eqref{eq:SL-intermediate}, then $T_0(A_1)$ is almost surely finite.
\end{lemma}

Before proving this result, we first explain briefly how it yields Proposition~\ref{prop:A} in combination with Lemma~\ref{lemma:properties-ODE}.

\begin{proof} [Proof of Proposition~\ref{prop:A}]
	The fact that $T_0(A_1)$ is almost surely finite is precisely the content of Lemma~\ref{lemma:properties-SDE}. Then by shifting $\mv{A}$ at time $T_0(A_1)$, we see that $\theta_{T_0(A_1)} \circ \mv{A}$ satisfies the ODE~\eqref{eq:ODE} with $y = 0$, and so the results of Lemma~\ref{lemma:properties-ODE} are precisely those that we need to prove for Proposition~\ref{prop:A}.
\end{proof}

\begin{proof} [Proof of Lemma~\ref{lemma:properties-SDE}]
	To prove $\P(T_0(A_1) < +\infty) = 1$, it is enough to prove that
	\begin{equation} \label{eq:goal-1}
		\P\left( T_0(A_1) < +\infty, T^\uparrow_{\gamma_1}(A_0) < +\infty \right) = \P\left( T^\uparrow_{\gamma_1}(A_0) < +\infty \right)
	\end{equation}
	and
	\begin{equation} \label{eq:goal-2}
		\P\left( T_0(A_1) = +\infty, T^\uparrow_{\gamma_1}(A_0) = +\infty \right) = 0
	\end{equation}
	which we do now.
	\\
	
	In the event $\{T^\uparrow_{\gamma_1}(A_0) < +\infty\}$, we have $\d \mv{A}^\uparrow(t) = b(\mv{A}^\uparrow(t)) \d t + \sigma(\mv{A}^\uparrow(t)) \d B^\uparrow(t)$, where $\mv{A}^\uparrow$ and $B^\uparrow$ are the processes $\mv{A}$ and $B$ shifted at time $T^\uparrow_{\gamma_1}(A_0)$, so that, since $T^\uparrow_{\gamma_1}(A_0)$ is a stopping time, $B^\uparrow$ is a Brownian motion according to the strong Markov property. Moreover, we have by definition $b_1(\mv{A}^\uparrow(t)) = (\gamma_1 - A_0^\uparrow(t)) A_1^\uparrow(t) \leq 0$ and so $A_1^\uparrow$ is dominated by Feller diffusion without drift (see, e.g., Theorem~V.$43.1$ in Rogers and Williams~\cite{Rogers87:0}), which almost surely hits $0$ in finite time. This proves~\eqref{eq:goal-1}.
	\\
	
	We now prove~\eqref{eq:goal-2}. Assume that $T^\uparrow_{\gamma_1}(A_0) = +\infty$: then $A_0$ is bounded (by $\gamma_1$, and in particular $\gamma_1 > 0$) and according to~\ref{ODE:>0}+\ref{ODE:boundedness} of Lemma~\ref{lemma:properties-ODE}, we obtain that $\int_0^\infty A_1$ is finite. In particular $L = 0$, where we define $L = \liminf_{t \to \infty} A_1(t)$, and so
	\[ \P\left( T_0(A_1) = +\infty,  T^\uparrow_{\gamma_1}(A_0) = +\infty \right) \leq \P\left( T_0(A_1) = +\infty, L = 0 \right). \]
	
	Thus to prove~\eqref{eq:goal-2}, we only have to show that $\P(T_0(A_1) = +\infty, L = 0) = 0$. Let $\vartheta_i$ be defined recursively by $\vartheta_0 = 0$ and $\vartheta_{i+1} = \inf\{ t \geq 1 + \vartheta_i: A_1(t) \leq 1 \}$. Then in $\{ L = 0 \}$, $\vartheta_i$ is almost surely finite for every $i \geq 0$ and so we can define $Z_i$ as the solution to the SDE $\d Z_i = \gamma_1 Z_i \d t + (2 Z_i)^{1/2} \d B_i$ with initial condition $Z_i(0) = 1$, where $B_i$ is the process $B$ shifted at time $\vartheta_i$. Note that conditionally on $\{\vartheta_i < +\infty\}$, $Z_i$ is a Feller diffusion with drift $\gamma_1 > 0$ started at $1$, and that the strong Markov property and the comparison theorem~V.$43.1$ in Rogers and Williams~\cite{Rogers87:0} show that $A_1(t + \vartheta_i) \leq Z_i(t)$ for $t \geq 0$ and $i \geq 1$. In particular,
	\[ \P \left(T_0(A_1) = +\infty, L = 0 \right) \leq \P \left(\forall i \geq 1: \vartheta_i < +\infty \text{ and } T_0(Z_i) = +\infty \right) \]
	and so we only have to show that this last probability is equal to $0$. For $I \geq 1$ and in the event $\{\vartheta_I < +\infty\}$, the strong Markov property at time $\vartheta_I$ gives
	\begin{align*}
		\P \big(\vartheta_I < +\infty \text{ and } T_0(Z_i) = +\infty \text{ for } i = 1, \ldots, I \big) & \\
		& \hspace{-40mm} = \P \big( T_0(Z_1) = +\infty \big) \P \big(\vartheta_I < +\infty \text{ and } T_0(Z_i) = +\infty \text{ for } i = 1, \ldots, I-1 \big)\\
		& \hspace{-40mm} \leq \P \big( T_0(Z_1) = +\infty \big) \P \big(\vartheta_{I-1} < +\infty \text{ and } T_0(Z_i) = +\infty \text{ for } i = 1, \ldots, I-1 \big)
	\end{align*}
	and so we obtain by induction
	\[ \P \big(\vartheta_I < +\infty \text{ and } T_0(Z_i) = +\infty \text{ for } i = 1, \ldots, I \big) \leq \left[ \P \big( T_0(Z_1) = +\infty \big) \right]^I. \]
	 
	Since $\P \big( T_0(Z_1) = +\infty \big) < 1$, letting $I \to +\infty$ achieves the proof of~\eqref{eq:goal-2}.
\end{proof}

\section{Scaling limits} \label{sec:proof-scaling-limit}

We now prove the convergence results of Theorem~\ref{thm:SL-intermediate} and Propositions~\ref{prop:SL-large} and~\ref{prop:SL-small}. The proofs of these three results can be cast into the same framework by defining $\mv{A}_n$ as in~\eqref{eq:scaling-large} with:
\begin{description}
	\item[Intermediate initial condition (Theorem~\ref{thm:SL-intermediate})] $\alpha_{n,1} = n^{1/(K+2)}$, $\tau_n = \alpha_{n,1}$, $\alpha_{n,0} = \alpha_{n,K+1}$ and $\alpha_{n,k} = \alpha_{n,1}^k$ for $k = 2, \ldots, K+1$;
	\item[Large initial condition (Proposition~\ref{prop:SL-large})] $n^{1/(K+2)} \ll \alpha_{n,1} \ll n$ and $\tau_n$ and $\alpha_{n,k}$ for $k = 2, \ldots, K+1$ are as in~\eqref{eq:scaling-constants-large};
	\item[Small initial condition (Proposition~\ref{prop:SL-small})] $1 \ll \alpha_{n,1} \ll n^{1/(K+2)}$, $\tau_n = \alpha_{n,1}$, $\alpha_{n,0} = \alpha_{n,K+1}$ and $\alpha_{n,k} = \alpha_{n,1}^k$ for $k = 2, \ldots, K+1$;
\end{description}
and by assuming, with this notation, that $\tau_n (\varepsilon_{n,k} - \delta_{n,k}) \to \gamma_k$, which is consistent with the assumptions~\eqref{eq:gamma-intermediate},~\eqref{eq:gamma-large} and~\eqref{eq:gamma-small}.
\\

In order to prove Theorem~\ref{thm:SL-intermediate} and Propositions~\ref{prop:SL-large} and~\ref{prop:SL-small} with this notation, we have to prove that the sequence $(\mv{A}_n, n \geq 1)$ converges weakly toward: the solution of~\eqref{eq:SL-intermediate} in the intermediate regime; the solution of~\eqref{eq:SL-large} in the large regime; and the solution of~\eqref{eq:SL-small} in the small regime. In the sequel we will use the fact that in all three regimes, the algebraic relations $\alpha_{n,0} = \alpha_{n,K+1}$ and $\alpha_{n,k} = \tau_n \alpha_{n,k-1}$ for $k = 2, \ldots, K+1$ hold.

The proof relies on the standard machinery: we first prove tightness and then identify accumulation points. Both steps rely on semimartingale arguments based on the explicit form of the generator of $\mv{A}_n$. Indeed, $\mv{A}_n$ is by definition a Markov process with generator $\Omega_n$ given, for any function $f: \R^{K+2} \to \R$ and any $a \in \R^{K+2}$, by
\begin{multline} \label{eq:Omega}
	\Omega_n(f)(a) = \tau_n \sum_{k=1}^K \left[ f\left(a-\frac{e_k}{\alpha_{n,k}} + \frac{e_{k+1}}{\alpha_{n,k+1}}\right) - f(a) \right] (1+\delta_{n,k}) \alpha_{n,k} a_k\\
	+ \tau_n \left( 1 - \alpha_{n,0} a_0 / n \right) \sum_{k=1}^K \left[ f\left(a+\frac{e_0}{\alpha_{n,0}}+\frac{e_k}{\alpha_{n,k}}\right) - f(a) \right] (1+\varepsilon_{n,k}) \alpha_{n,k} a_k
\end{multline}
(note that $a_{n,0}/n = 1 - \alpha_{n,0} A_{n,0} / n$, which gives the factor $1 - \alpha_{n,0} a_0 / n$ in the above expression). Since $\mv{A}_n$ lives on a finite state space, for any function $f$ the process
\[ M_n^f(t) = f(\mv{A}_n(t)) - f(\mv{A}_n(0)) - V_n^f(t) \ \text{ with } \ V_n^f(t) = \int_0^t \Omega_n(f)(\mv{A}_n(u)) \d u \]
is a martingale whose quadratic variation process $\langle M_n^f \rangle$ is equal to
\[ \langle M_n^f \rangle(t) = \int_0^t \Bcal_n(f)(\mv{A}_n(u)) \d u \ \text{ with } \ \Bcal_n(f) = \Omega_n(f^2) - 2 f \Omega_n(f), \]
see for instance Lemma~VII.$3.68$ in Jacod and Shiryaev~\cite{Jacod03:0}. Recall that $\delta_{n,0} = \delta_{n,K+1} = \varepsilon_{n,K+1} = -1$, and let in the rest of this section $\pi_k$ for $k = 0, \ldots, K+1$ be the projection on the $k$th coordinate, i.e., $\pi_k(a) = a_k$ for any $a = (a_0, \ldots, a_{K+1}) \in \R^{K+2}$.

\begin{lemma} \label{lemma:expression-V-M}
	For each $n \geq 1$ and $t \geq 0$,
	\begin{equation} \label{eq:V_0}
		V_n^{\pi_0}(t) = \frac{\tau_n}{\alpha_{n,0}} \sum_{k=1}^K (1+\varepsilon_{n,k}) \alpha_{n,k} \int_0^t (1 - \alpha_{n,0} A_{n,0}(u) / n) A_{n,k}(u) \d u,
	\end{equation}
	\begin{multline} \label{eq:V_k}
		V_n^{\pi_k}(t) = \int_0^t \Big[ (1+\delta_{n,k-1}) A_{n,k-1}(u) + \tau_n (\varepsilon_{n,k} - \delta_{n,k}) A_{n,k}(u)\\
		\left. - \frac{\tau_n \alpha_{n,0}}{n} (1+\varepsilon_{n,k}) A_{n,0}(u) A_{n,k}(u) \right] \d u
	\end{multline}
	for $k = 1, \ldots, K+1$ and
	\begin{equation} \label{eq:M}
		\langle M_n^{\pi_k} \rangle(t) = \frac{1}{\alpha_{n,k}} V_n^{\pi_k}(t) + 2 \frac{\tau_n}{\alpha_{n,k}} (1+\delta_{n,k}) \int_0^t A_{n,k}(u) \d u
	\end{equation}
	for $k = 0, \ldots, K+1$.
\end{lemma}

\begin{proof}
	Recall that $V_n^{\pi_k}(t) = \int_0^t \Omega_n(\pi_k)(\mv{A}_n(u)) \d u$ and that $\langle M_n^{\pi_k} \rangle(t) = \int_0^t \Bcal_n(\pi_k)(\mv{A}_n(u)) \d u$, so that we only have to compute $\Omega_n(\pi_k)$ and $\Bcal_n(\pi_k)$ for $k=0, \ldots, K+1$. By writing $\Omega_n(f^2)(a) - 2 f(a) \Omega_n(f)(a)$ in the form
	\[ \sum_j (x_j^2 - y^2) \beta_j - 2 y \sum_j(x_j - y) \beta_j = \sum_j \left( (x_j^2 - y^2) - 2 y (x_j - y) \right) \beta_j = \sum_j (x_j - y)^2 \beta_j \]
	for some $x_j$, $\beta_j$ and $y$, we see that $\Bcal_n(f)$ can alternatively be written as follows:
	\begin{multline} \label{eq:bcal}
		\Bcal_n(f)(a) = \tau_n \sum_{k=1}^K \left[ f\left(a-\frac{e_k}{\alpha_{n,k}} + \frac{e_{k+1}}{\alpha_{n,k+1}}\right) - f(a) \right]^2 (1+\delta_{n,k}) \alpha_{n,k} a_k\\
		+ \tau_n \left( 1 - \alpha_{n,0} a_0/ n \right) \sum_{k=1}^K \left[ f\left(a+\frac{e_0}{\alpha_{n,0}}+\frac{e_k}{\alpha_{n,k}}\right) - f(a) \right]^2 (1+\varepsilon_{n,k}) \alpha_{n,k} a_k.
	\end{multline}
	
	For $k = 0$,
	\[ \Omega_n(\pi_0)(a) = \frac{\tau_n}{\alpha_{n,0}}(1 - \alpha_{n,0} a_0/n) \sum_{k=1}^K (1+\varepsilon_{n,k}) \alpha_{n,k} a_k \ \text{ and } \ \Bcal_n(\pi_0)(a) = \frac{1}{\alpha_{n,0}} \Omega_n(\pi_0)(a) \]
	which proves the result for $k = 0$. For $k = K+1$,
	\[ \Omega_n(\pi_{K+1})(a) = \frac{\tau_n}{\alpha_{n,K+1}} (1+\delta_{n,K}) \alpha_{n,K} a_K = (1+\delta_{n,K}) a_K \]
	using $\alpha_{n,K+1} = \tau_n \alpha_{n,K}$ to get the last equality, and $\Bcal_n(\pi_{K+1}) = \Omega_n(\pi_{K+1}) / \alpha_{n,K+1}$, which proves the result for $k = K+1$. Consider now $k =1, \ldots, K$:
	\begin{align*}
		\Omega_n(\pi_k)(a) & = \tau_n \left( -\frac{1}{\alpha_{n,k}} (1+\delta_{n,k}) \alpha_{n,k} a_k + \frac{1}{\alpha_{n,k}} (1+\delta_{n,k-1}) \alpha_{n,k-1} a_{k-1} \right)\\
		& \hspace{55mm} + \frac{\tau_n}{\alpha_{n,k}} (1-\alpha_{n,0} a_0 / n) (1+\varepsilon_{n,k}) \alpha_{n,k} a_k\\
		& = (1+\delta_{n,k-1}) a_{k-1} + \tau_n (\varepsilon_{n,k} - \delta_{n,k}) a_k - \frac{\tau_n \alpha_{n,0}}{n} (1+\varepsilon_{n,k}) a_0 a_k,
	\end{align*}
	using $\delta_{n,0} = -1$ for $k = 1$, and $\alpha_{n,k-1} \tau_n = \alpha_{n,k}$ for $k \geq 2$. This proves the result for $V_n^{\pi_k}$, while writing
	\begin{align*}
		\Bcal_n(\pi_k)(a) & = \tau_n \left( \frac{1}{\alpha_{n,k}^2} (1+\delta_{n,k}) \alpha_{n,k} a_k + \frac{1}{\alpha_{n,k}^2} (1+\delta_{n,k-1}) \alpha_{n,k-1} a_{k-1} \right)\\
		& \hspace{55mm} + \frac{\tau_n}{\alpha_{n,k}^2} (1-\alpha_{n,0} a_0 / n) (1+\varepsilon_{n,k}) \alpha_{n,k} a_k\\
		& = \frac{1}{\alpha_{n,k}} \Omega_n(\pi_k)(a) + 2 \frac{\tau_n}{\alpha_{n,k}} (1 + \delta_{n,k}) a_k
	\end{align*}
	proves the result for $M_n^{\pi_k}$, which concludes the proof.
\end{proof}

\subsection{Tightness} We now prove that the sequence $(\mv{A}_n, n \geq 1)$ is tight. Since $\mv{A}_n$ makes jumps of vanishing size (at most $1 / \alpha_{n,1}$), in order to show that $(\mv{A}_n, n \geq 1)$ is tight it is sufficient to show that for every $T \geq 0$,
\begin{equation} \label{eq:condition-tightness}
	\lim_{\eta \to 0} \limsup_{n \to +\infty} \Delta_n(\eta) = 0 \ \text{ where } \ \Delta_n(\eta) = \sup_\Upsilon \sup_{0 \leq t \leq \eta} \E \left( \sum_{k = 0}^{K+1} \norm{A_{n,k}(\Upsilon + t) - A_{n,k}(\Upsilon)}  \right)
\end{equation}
and where the first supremum in the definition of $\Delta_n(\eta)$ is taken over all the random variables $\Upsilon \leq T$ that are stopping times relatively to the filtration generated by $\mv{A}_n$ (see Corollary on page~$179$ in Billingsley~\cite{Billingsley99:0}). Fix in the rest of the proof some $T \geq 0$. Because of the strong Markov property and the fact that we consider stopping times $\Upsilon \leq T$, we have
\begin{equation} \label{eq:tightness-2}
	\Delta_n(\eta) \leq \sup_{0 \leq t \leq \eta} \E \left[ \sup_{0 \leq y \leq T} \Phi_n^t(\mv{A}_n(y)) \right]
\end{equation}
where
\begin{equation} \label{eq:Phi}
	\Phi_n^t(a) = \sum_{k = 0}^{K+1} \E_a \left( \norm{A_{n,k}(t) - A_{n,k}(0)} \right)
\end{equation}
and the subscript refers to the initial state of the process $\mv{A}_n$ (when there is no subscript, this refers to an initial condition as in the statement of the theorem or the propositions). By definition we have $\E_a(\lvert A_{n,k}(t) - A_{n,k}(0) \rvert) = \E_a(\lvert M_n^{\pi_k}(t) + V_n^{\pi_k}(t)\rvert)$, and so combining the triangular inequality with Cauchy-Schwarz inequality and summing over $k = 0, \ldots, K+1$ leads to
\begin{equation} \label{eq:CS}
	\Phi_n^t(a) \leq \sum_{k=0}^{K+1} \sqrt{ \E_a \left( \langle M_n^{\pi_k} \rangle(t) \right)} + \sum_{k=0}^{K+1} \E_a \left( \norm{V_n^{\pi_k}(t)} \right).
\end{equation}

In order to do control the first moment of $\langle M_n^{\pi_k} \rangle(t)$ and of $V_n^{\pi_k}(t)$, we introduce the functions $\psi = \pi_0 + \cdots + \pi_K$, i.e., $\psi(a) = a_0 + \cdots + a_K$ for $a = (a_0, \ldots, a_{K+1}) \in \R^{K+2}$, and $\Psi = \psi + \psi^2$. Since $\alpha_{n,1} \to +\infty$ as $n \to +\infty$, $\alpha_{n,k+1} = \tau_n \alpha_{n,k}$ for $k = 1, \ldots, K$ and $\tau_n \to +\infty$, we will assume for convenience that $1 \leq \alpha_{n, 1} \leq \cdots \leq \alpha_{n,K+1}$. Moreover, it can be checked that the constant
\[ C_0 = \sup_{n \geq 1, 1 \leq k \leq K} \left( \tau_n \lvert \varepsilon_{n,k} - \delta_{n,k} \rvert, 1 + \lvert \delta_{n,k} \rvert, 1 + \lvert \varepsilon_{n,k} \rvert, \frac{\tau_n \alpha_{n,0}}{n}, \frac{2 \tau_n}{\alpha_{n,k}}, \frac{\tau_n \alpha_{n,k}}{\alpha_{n,0}} \right) \]
is finite. From these various definitions and~\eqref{eq:V_0}--\eqref{eq:M}, and also recalling that $\delta_{n,0} = \delta_{n,K+1} = \varepsilon_{n,K+1} = -1$, it follows that
\begin{equation} \label{eq:bound-increments}
	\max \left( \norm{V_n^{\pi_k}(t)}, \ \langle M_n^{\pi_k} \rangle(t) \right) \leq 2 C_0^2 \int_0^t \Psi(\mv{A}_n(u)) \d u, \ k = 0, \ldots, K+1, t \geq 0.
\end{equation}

Combined with~\eqref{eq:CS} this leads to
\[ \Phi_n^t(a) \leq C_1 \sqrt{\int_0^t \E_a \left[ \Psi(\mv{A}_n(u)) \right] \d u} + C_1 \int_0^t \E_a \left[ \Psi(\mv{A}_n(u)) \right] \d u \]
with $C_1 = 2 C_0^2 (K+2)$. For $t \leq \eta \leq 1$ we obtain
\begin{equation} \label{eq:interm-1}
	\Phi_n^t(a) \leq C_1 \eta^{1/2} \left( \sup_{0 \leq u \leq 1} \sqrt{\E_a \left[ \Psi(\mv{A}_n(u)) \right]} + \sup_{0 \leq u \leq 1} \E_a \left[ \Psi(\mv{A}_n(u)) \right] \right).
\end{equation}

\begin{lemma}\label{lemma:bound-psi}
	There exists a finite constant $C_2$ such that for every initial state $a$,
	\begin{equation} \label{eq:bound-psi}
		\sup_{0 \leq u \leq 1} \E_a \left[ \psi^i(\mv{A}_n(u)) \right] \leq C_2 \sum_{j=0}^i \psi^j(a), \ i = 1,2.
	\end{equation}
	
	In particular,
	\begin{equation} \label{eq:bound-Psi}
		\sup_{n \geq 1} \ \E \left[ \sup_{0 \leq y \leq T} \Psi(\mv{A}_n(y)) \right] < +\infty.
	\end{equation}
\end{lemma}

Let us quickly finish the proof of the tightness of $(\mv{A}_n)$ based on this lemma: first, plugging in~\eqref{eq:bound-psi} into~\eqref{eq:interm-1}, we obtain the existence of a finite constant $C_3$ such that
\[ \Phi_n^t(a) \leq C_3 \eta^{1/2} \left( \sqrt{\Psi(a)} + \Psi(a) \right), \ 0 \leq t \leq \eta \leq 1, \]
and so in view of~\eqref{eq:tightness-2}, we see that for every $\eta \leq 1$,
\[ \Delta_n(\eta) \leq C_3 \eta^{1/2} \E \left[ \sup_{0 \leq y \leq T} \left( \sqrt{\Psi(\mv{A}_n(y))} + \Psi(\mv{A}_n(y)) \right) \right]. \]

Thus,~\eqref{eq:bound-Psi} implies the existence of a finite constant $C_4$ such that $\Delta_n(\eta) \leq C_4 \eta^{1/2}$ which achieves to prove that $(\mv{A}_n)$ is tight. It remains to prove Lemma~\ref{lemma:bound-psi}.

\begin{proof} [Proof of Lemma~\ref{lemma:bound-psi}]
	Let us first prove~\eqref{eq:bound-psi}. Defining $\eta_{n,k} = 1/\alpha_{n,k} - \indicator{k \neq K}/\alpha_{n,k+1} \geq 0$ and writing
	\[ \psi^i(a) - \psi^i \left( a - \frac{e_k}{\alpha_{n,k}} + \frac{e_{k+1}}{\alpha_{n,k+1}} \right) = \left( \psi(a) \right)^i - \left( \psi(a) - \eta_{n,k} \right)^i = i \int_{\psi(a) - \eta_{n,k}}^{\psi(a)} x^{i-1} \d x, \]
	we get
	\[ \psi^i(a) - \psi^i \left( a - \frac{e_k}{\alpha_{n,k}} + \frac{e_{k+1}}{\alpha_{n,k+1}} \right) \geq i \eta_{n,k} \left( \psi(a) - \eta_{n,k} \right)^{i-1}. \]

	Defining $\mu_{n,k} = 1/\alpha_{n,0} + 1/\alpha_{n,k}$ similarly leads to
	\[ \psi^i \left( a + \frac{e_0}{\alpha_{n,0}} + \frac{e_k}{\alpha_{n,k}} \right) - \psi^i(a) \leq i \mu_{n,k} \left( \psi(a) + \mu_{n,k} \right)^{i-1}. \]
	
	Plugging these inequalities in the definition~\eqref{eq:Omega} of $\Omega_n$, we obtain
	\begin{multline*}
		\Omega_n(\psi^i)(a) \leq i \tau_n (1-\alpha_{n,0} a_0 / n) \sum_{k=1}^K \mu_{n,k} \left( \psi(a) + \mu_{n,k} \right)^{i-1} (1+\varepsilon_{n,k}) \alpha_{n,k} a_k\\
		- i \tau_n \sum_{k=1}^K \eta_{n,k} \left( \psi(a) - \eta_{n,k} \right)^{i-1} (1+\delta_{n,k}) \alpha_{n,k} a_k.
	\end{multline*}
	
	Using $1 - \alpha_{n,0} a_0 / n \leq 1$, expanding the terms raised to the power $i{-}1$ and changing the order of summation, we end up with
	\begin{multline*}
		\Omega_n(\psi^i)(a) \leq i \sum_{j=0}^{i-1} \left( \begin{array}{c} i-1 \\ j \end{array} \right) \left(\psi(a)\right)^{i-1-j}\\
		\times \left\{ \tau_n \sum_{k=1}^K \mu_{n,k}^{j+1} (1+\varepsilon_{n,k}) \alpha_{n,k} a_k - (-1)^j \tau_n \sum_{k=1}^K \eta_{n,k}^{j+1} (1+\delta_{n,k}) \alpha_{n,k} a_k \right\}.
	\end{multline*}
	
	Plugging in the definitions of $\eta_{n,k}$ and $\mu_{n,k}$, we see that for $j = 0$ the term between brackets in the previous display is given by
	\begin{multline*}
		\tau_n \sum_{k=1}^K (1+\varepsilon_{n,k}) (1 + \alpha_{n,k} / \alpha_{n, 0}) a_k - \tau_n \sum_{k=1}^K (1+\delta_{n,k}) (1 - \indicator{k \not = K} \alpha_{n,k} / \alpha_{n, k+1}) \alpha_{n,k} a_k\\
		= \tau_n \sum_{k=1}^K (\varepsilon_{n,k} - \delta_{n,k}) a_k + \sum_{k=1}^K \frac{\tau_n \alpha_{n,k}}{\alpha_{n,0}} (1 + \varepsilon_{n,k}) a_k + \sum_{k=1}^{K-1} (1 + \delta_{n,k}) a_k \leq 3 C_0^2 \psi(a).
	\end{multline*}

	For $j \geq 1$, we obtain similarly, using also $\mu_{n,k} \leq 2 / \alpha_{n,k} \leq 2$ and $\eta_{n,k} \leq 1 / \alpha_{n,k} \leq 1$, that the term between brackets is upper bounded by
	\[ \tau_n \sum_{k=1}^K \frac{4}{\alpha_{n,k}^2} (1+\varepsilon_{n,k}) \alpha_{n,k} a_k + \tau_n \sum_{k=1}^K \frac{1}{\alpha_{n,k}^2} (1+\delta_{n,k}) \alpha_{n,k} a_k \leq 3 C_0^2 \psi(a). \]
	
	We thus obtain, for some finite constant $C(i)$ and every $a$,
	\begin{equation} \label{eq:bound-Omega(psi^i)}
		\Omega_n(\psi^i)(a) \leq C(i) \sum_{j=1}^i \left(\psi(a)\right)^j.
	\end{equation}
		
	In particular,
	\begin{align*}
		\E_a \left( \psi^i(\mv{A}_n(t)) \right) & = \psi^i(\mv{A}_n(0)) + \int_0^t \E_a \left( \Omega_n (\psi^i) (\mv{A}_n(u)) \right) \d u\\
		& \leq C(i) \int_0^t \E_a \left(\sum_{j=0}^{i-1}\psi^j(\mv{A}_n(u))\right) \d u + C(i) \int_0^t \E_a \left(\psi^i(\mv{A}_n(u))\right) \d u
	\end{align*}
	assuming without loss of generality that $C(i) \geq 1$ for the last inequality. Thus Gronwall's lemma implies
	\[ \E_a \left( \psi^i(\mv{A}_n(t)) \right) \leq C(i) \int_0^t \sum_{j=0}^{i-1} \E_a \left( \psi^j(\mv{A}_n(u)) \right) \d u \times e^{C(i)t}. \]
	
	Then~\eqref{eq:bound-psi} follows from this inequality by induction on $i$. We now derive~\eqref{eq:bound-Psi}: since $\Psi = \psi + \psi^2$ it is enough to prove~\eqref{eq:bound-Psi} with $\psi^2$ in place of $\Psi$. First of all, note that the previous reasoning shows the existence of a finite constant $C'_2$ such that
	\begin{equation} \label{eq:bound-psi'}
		\sup_{0 \leq u \leq T} \E \left[ \psi^i(\mv{A}_n(u)) \right] \leq C_2' \sum_{j=0}^i \psi^j(\mv{A}_n(0)), \ i = 1,2,3,4.
	\end{equation}
	
	By definition, $\psi^2 \circ \mv{A}_n = \psi^2(\mv{A}_n(0)) + V_n + M_n$, defining $V_n = V_n^{\psi^2}$ and $M_n = M_n^{\psi^2}$, and so
	\[ \E \left[ \sup_{0 \leq y \leq T} \psi^2(\mv{A}_n(y)) \right] \leq \psi^2(\mv{A}_n(0)) + \E \left[ \sup_{0 \leq y \leq T} V_n(y) \right] + \E \left[ \sup_{0 \leq y \leq T} M_n(y) \right]. \]
	
	The first term of the above upper bound is bounded (in $n$) since the sequence $\mv{A}_n(0)$ converges by assumption. For the second term, we write
	\[ V_n(y) = \int_0^y \Omega_n(\psi^2)(\mv{A}_n(u)) \d u \leq C(2) \int_0^y \Psi(\mv{A}_n(u)) \d u \]
	where the inequality comes from~\eqref{eq:bound-Omega(psi^i)}. Thus
	\[ \E \left[ \sup_{0 \leq y \leq T} V_n(y) \right] \leq C(2) \int_0^T \E \big[ \Psi(\mv{A}_n(u)) \big] \d u \leq C(2) T \sup_{0 \leq u \leq T} \E \big[ \Psi(\mv{A}_n(u)) \big] \]
	which is finite by~\eqref{eq:bound-psi'}. We now control the last martingale term. For any real-valued random variable $X$ we have $\E(X) \leq 1 + \E(X^2)$ and so Doob's inequality gives
	\[ \E \left( \sup_{0 \leq t \leq T} M_n(t) \right) \leq 1 + 4 \E \left[ \langle M_n(T) \rangle \right] = 1 + 4 \int_0^T \E \left[ \Bcal_n(\psi^2) (\mv{A}_n(u)) \right] \d u. \]

	Defining $\Bcal'_n$ as in~\eqref{eq:bcal} but with the term $- \alpha_{n,0} a_0 / n$ taken equal to $0$, we have
	\[ \Bcal_n(\psi^2)(a) \leq \Bcal_n'(\psi^2)(a) = \Omega'_n(\psi^4)(a) - 2 \psi^2(a) \Omega'_n(\psi^2)(a) \]
	defining $\Omega'_n$ similarly as $\Omega_n$ in~\eqref{eq:Omega} but with the term $- \alpha_{n,0} a_0 / n$ taken equal to $0$. Since the first step in the derivation of~\eqref{eq:bound-Omega(psi^i)} was to use $1 - \alpha_{n,0} a_0 / n \leq 1$, the reasoning leading to~\eqref{eq:bound-Omega(psi^i)} also leads to an upper bound on $\lvert \Omega'_n(\psi^2) \rvert$ of the same kind, i.e., it leads to the existence of a finite constant $C'(i)$ such that $\lvert \Omega'_n(\psi^i)(a) \rvert \leq C'(i) \sum_{j=1}^i \psi^j(a)$. This finally proves that $\Bcal_n(\psi^2)(a) \leq C' \sum_{j=1}^4 \psi^j(a)$ for some finite constant $C'$, and in particular
	\[ \E \left( \sup_{0 \leq t \leq T} M_n(t) \right) \leq 1 + 4 C' T \sum_{j=1}^4 \sup_{0 \leq u \leq T} \E \left[ \psi^j (\mv{A}_n(u)) \right]. \]
	
	Since the supremum over $n \geq 1$ of the right-hand side is finite by~\eqref{eq:bound-psi'}, the proof is complete.
\end{proof}

\subsection{Characterization of accumulation points}

Let $\mv{A} = (A_k, k = 0, \ldots, K+1)$ be any accumulation point of $(\mv{A}_n)$ and assume without loss of generality that $\mv{A}_n \Rightarrow \mv{A}$, i.e., $\mv{A}_n$ converges weakly to $\mv{A}$. Remember that we have to prove that $\mv{A}$ satisfies~\eqref{eq:SL-intermediate} in the intermediate regime,~\eqref{eq:SL-large} in the large regime and~\eqref{eq:SL-small} in the small regime. We treat the three regimes separately.

In each regime, if $f \in \Dcal$ we denote by $\int f$ the function $(\int_0^t f(u) \d u, t \geq 0)$, and we will use the following result in conjunction with the continuous mapping theorem. If $(f_n)$, $(g_n)$ are two sequences of functions with $f_n \to f$ and $g_n \to g$ for some continuous functions $f$ and $g$ (where $\to$ denotes uniform convergence on compact sets), then $f_n g_n \to f g$ and $\int f_n \to \int f$.

In particular, since in all three regimes it holds that $\tau_n \alpha_{n,k} / \alpha_{n,0} = \alpha_{n,k+1} / \alpha_{n,K+1} \to \indicator{k = K}$ and $\alpha_{n,0} / n \to 0$, we have from~\eqref{eq:V_0} that $V_n^{\pi_0} \Rightarrow \int A_K$. Since $\alpha_{n,0} \to +\infty$ and $\delta_{n,0} = -1$ we have $\langle M^{\pi_0}_n \rangle \Rightarrow 0$ by~\eqref{eq:M}, which implies by Doob's inequality that $M^{\pi_0}_n \Rightarrow 0$. Similarly, recalling that $\delta_{n,K+1} = \varepsilon_{n,K+1} = -1$ we have $V_n^{\pi_{K+1}} \Rightarrow \int A_K$ and $M_n^{\pi_{K+1}} \Rightarrow 0$.

\subsubsection{Large initial condition} In this regime, we have $\tau_n \alpha_{n,0} = n$. Since moreover $\delta_{n,0} = \delta_{n, K+1} = \varepsilon_{n,K+1} = -1$, we obtain from~\eqref{eq:V_k} that $V_n^{\pi_1} \Rightarrow \int (\gamma_1 - A_0) A_1$ and $V_n^{\pi_k} \Rightarrow \int (A_{k-1} + (\gamma_k - A_0) A_k)$ for $k = 2, \ldots, K$. In other words, $V_n^{\pi_k} \Rightarrow \int b_k \circ \mv{A}$ for $k = 0, \ldots, K+1$.

Moreover, since $\alpha_{n,k} \to 0$ and $\tau_n / \alpha_{n,k} \to 0$ for $k = 1, \ldots, K$, we obtain from~\eqref{eq:M} that $\langle M_n^{\pi_k} \rangle \Rightarrow 0$ for any $k = 0, \ldots, K+1$. By Doob's inequality, this implies $M_n^{\pi_k} \Rightarrow 0$ and since $A_{n,k} = A_{n,k}(0) + M_n^{\pi_k} + V_n^{\pi_k}$ we finally get $A_{n,k} \Rightarrow a_k + \int b_k \circ \mv{A}$. Since all the above convergences hold jointly as a consequence of the continuous mapping theorem, we get on the one hand that $\mv{A}_n \Rightarrow a + \int b \circ \mv{A}$, while on the other hand, since $\mv{A}_n \Rightarrow \mv{A}$ by assumption, we get $\mv{A} = a + \int b \circ \mv{A}$. Thus $\mv{A}$ solves~\eqref{eq:SL-large} as desired, and since uniqueness of solutions to this ODE is guaranteed by Lemma~\ref{lemma:stab-ODE}, this uniquely characterizes $\mv{A}$.

\subsubsection{Intermediate initial condition} We still have $\tau_n \alpha_{n,0} = n$, so that as in the large regime we have $V_n^{\pi_k} \Rightarrow \int b_k \circ \mv{A}$. Moreover, in this regime we have $\tau_n / \alpha_{n,k} \to 0$ for $k = 2, \ldots, K$ which implies as in the previous regime $M_n^{\pi_k} \Rightarrow 0$ for $k = 0, 2, \ldots, K+1$.

The difference with the large regime is that since $\tau_n = \alpha_{n,1}$, we have $\langle M_n^{\pi_1} \rangle \Rightarrow 2 \int A_1$. Note that $2 \int A_1 = \langle M \rangle$ where $M(t) = \int_0^t (2A_1(u))^{1/2} \d B(u)$. Moreover, since $M_n^{\pi_k} \Rightarrow 0$ for $k \neq 1$ while $\langle M_n^{\pi_1} \rangle \Rightarrow \langle M \rangle$, we get that the quadratic co-variation processes $\langle M_n^{\pi_k}, M_n^{\pi_\ell} \rangle$ vanish for any $k \neq \ell$ by polarization. Since all these convergences hold jointly, Theorem IX.$2$.$4$ in Jacod and Shiryaev~\cite{Jacod03:0} shows that $\mv{A}$ is the semimartingale with characteristics $(b \circ \mv{A}, M)$ in the sense that $\mv{A} = a + \int b \circ \mv{A} + M$, i.e., $\mv{A}$ solves~\eqref{eq:SL-intermediate} (and thus is uniquely determined by Lemma~\ref{lemma:SDE-exact}).

%

\subsubsection{Small initial condition} It follows similarly as in the two previous regimes, noting that in this regime we have $\tau_n \alpha_{n,0} / n \to 0$ (which leads to the drift term $b^S$ instead of $b$) and $\tau_n / \alpha_{n,k} \to \indicator{k = 1}$ (which leads to the non-vanishing diffusion term as in the intermediate regime).

\section{Asymptotic behavior of the outbreak size} \label{sec:proof-outbreak}

This section is devoted to proving Theorem~\ref{thm:outbreak}. For $f \in \Dcal$ we define the operator $\bar T_0(f) = \sup\{ t \geq 0: f(t) > 0 \}$; recall also the various operators defined in Section~\ref{sub:notation}. Then we have $A_{n,0}(\infty) = A_{n,0}(\bar T_0(A_{n,K}))$ and, in order to compute $\bar T_0(A_{n,K})$, we will use the relation $\bar T_0(A_{n,2}) = T_0(A_{n,1}) + T_0 \circ \theta_{T_0(A_{n,1})} (A_{n,2})$ which, iterated, leads to
\[ \bar T_0(A_{n,K}) = \sum_{k=1}^{K-1} T_0 \circ \theta_{T_0(A_{n,k-1})} \circ \cdots \circ \theta_{T_0(A_{n,1})} (A_{n,k}). \]

There are two difficulties to solve in order to prove Theorem~\ref{thm:outbreak}: the first one is that hitting times are in general not continuous functional, i.e., we may have $f_n \to f$ but $T_0(f_n) \not \to T_0(f)$. The second difficulty is that, as Proposition~\ref{thm:SL-intermediate} shows, $\bar T_0(A_K) = +\infty$ for $K \geq 2$ while the convergence $\mv{A}_n \Rightarrow \mv{A}$ holds uniformly on \emph{compact} sets (see the discussion following Theorem~\ref{thm:outbreak}). We address these two difficulties in two steps.

\subsection{First step} The goal of this first step is to prove that $(\mv{A}_n, T_0(A_{n,1})) \Rightarrow (\mv{A}, T_0(A_1))$. Assume first that for every $\delta > 0$,
\begin{equation} \label{eq:limsup}
	\limsup_{n \to +\infty} \P \left( T_0(A_{n,1}) - T^\downarrow_\varepsilon(A_{n,1}) \geq \delta \right) \mathop{\longrightarrow}_{\varepsilon \to 0} 0.
\end{equation}

We now argue that this implies $(\mv{A}_n, T_0(A_{n,1})) \Rightarrow (\mv{A}, T_0(A_1))$. First of all, note that if $T_0(A_{n,1}) \Rightarrow T_0(A_1)$ then the joint convergence automatically holds, see for instance Corollary~$2.2$ in Lambert et al.~\cite{Lambert13:0}.

To see that~\eqref{eq:limsup} implies $T_0(A_{n,1}) \Rightarrow T_0(A_1)$, let us say that $A_1$ \emph{goes across} $\varepsilon$ if for every $\eta > 0$ we have $\inf_{0 \leq t \leq \eta} A_1(T^\downarrow_\varepsilon(A_1) + t) < \varepsilon$, and define the random set $\Gcal$ through $\Gcal = \{ \varepsilon > 0: A_1 \text{ goes across } \varepsilon \}$. Then it is known (and actually easy to show) that if $T^\downarrow_\varepsilon(A_1)$ is almost surely finite and $\P(\varepsilon \in \Gcal) = 1$, then $T^\downarrow_\varepsilon(A_{n,1}) \Rightarrow T^\downarrow_\varepsilon(A_1)$, see for instance Proposition~VI.$2$.$11$ in Jacod and Shiryaev~\cite{Jacod03:0} or Lemma~$3.1$ in Lambert and Simatos~\cite{Lambert+:2}. Note that in our case, $T^\downarrow_\varepsilon(A_1)$ is finite by Proposition~\ref{prop:A}.

On the other hand, the complement $\Gcal^c$ of $\Gcal$ is precisely the set of discontinuities of the process $(T^\downarrow_\varepsilon(A_1), \varepsilon > 0)$. Since $(T^\downarrow_\varepsilon(A), \varepsilon > 0)$ is c\`agl\`ad, as the left-continuous inverse of the process $(\inf_{[0,t]} A_1, t \geq 0)$, the set $\{ \varepsilon > 0: \P(\varepsilon \in \Gcal^c) > 0\}$ (sometimes called set of fixed times of discontinuities) is at most countable, see for instance Billingsley~\cite[Section $13$]{Billingsley99:0}. Gathering these two observations, we see that $T^\downarrow_\varepsilon(A_{n,1}) \Rightarrow T^\downarrow_\varepsilon(A_1)$ for all $\varepsilon > 0$ outside a countable set. Then, writing
\begin{multline*}
	\P \left( T_0(A_{n,1}) \geq x \right) = \P \left( T_0(A_{n,1}) \geq x, T_0(A_{n,1}) - T^\downarrow_\varepsilon(A_{n,1}) \geq \delta \right)\\
	+ \P \left( T_0(A_{n,1}) \geq x, T_0(A_{n,1}) - T^\downarrow_\varepsilon(A_{n,1}) < \delta \right),
\end{multline*}
using~\eqref{eq:limsup} and playing with quantifiers gives the convergence of $T_0(A_{n,1})$ toward $T_0(A_1)$. Indeed, we can for instance write
\[ \P \left( T_0(A_{n,1}) \geq x \right) \leq \P \left( T_0(A_{n,1}) - T^\downarrow_\varepsilon(A_{n,1}) \geq \delta \right) + \P \left( T^\downarrow_\varepsilon(A_{n,1}) \geq x - \delta \right), \]
then choose $\varepsilon$ such that $T^\downarrow_\varepsilon(A_{n,1}) \to T^\downarrow_\varepsilon(A_1)$ to get by the portmanteau theorem, for any $\delta > 0$,
\[ \limsup_{n \to +\infty} \P \left( T_0(A_{n,1}) \geq x \right) \leq \limsup_{n \to +\infty} \P \left( T_0(A_{n,1}) - T^\downarrow_\varepsilon(A_{n,1}) \geq \delta \right) + \P \left( T^\downarrow_\varepsilon(A_1) \geq x - \delta \right). \]

Since $T^\downarrow_\varepsilon(A_1) \to T_0(A_1)$ as $\varepsilon \to 0$, we get by letting first $\varepsilon \to 0$ and then $\delta \to 0$, and using~\eqref{eq:limsup},
\[ \limsup_{n \to +\infty} \P \left( T_0(A_{n,1}) \geq x \right) \leq \P \left( T_0(A_1) \geq x \right). \]

Since $x$ was arbitrary, this shows that $T_0(A_{n,1}) \Rightarrow T_0(A_1)$ by the portmanteau theorem. In conclusion,~\eqref{eq:limsup} indeed implies $T_0(A_{n,1}) \Rightarrow T_0(A_1)$.
\\

The proof of~\eqref{eq:limsup} relies on a simple coupling between $A_{n,1}$ and a continuous-time branching process (more precisely, a Bellman-Harris branching process). Looking at the transition rates of the process $a_{n,1}$ in~\eqref{eq:transition-rates-a_n}, we see that $A_{n,1}$ decreases by $1/\alpha_{n,1}$ at rate $(1+\delta_{n,1}) A_{n,1} \alpha_{n,1}^2$ and increases by $1/\alpha_{n,1}$ at rate
\[ (1 + \varepsilon_{n,1}) A_{n,1} (1-A_{n,0} / \alpha_{n,1}) \alpha_{n,1}^2 \leq (1 + \varepsilon_{n,1}) A_{n,1} \alpha_{n,1}^2. \]

In particular, shifting the origin of time at $T^\downarrow_\varepsilon(A_{n,1})$ and using the strong Markov property, one sees that $A_{n,1}$ can be coupled with a Markov process $Z_{n,1}$ in such a way that $Z_{n,1}(0) = \alpha_{n,1} \lfloor \varepsilon / \alpha_{n,1} \rfloor$, $A_{n,1}(T^\downarrow_\varepsilon(A_1) + t) \leq Z_{n,1}(t)$ for $t \geq 0$, and $Z_{n,1}$ decreases by $1/\alpha_{n,1}$ at rate $(1+\delta_{n,1}) Z_{n,1} \alpha_{n,1}^2$ and increases by $1/\alpha_{n,1}$ at rate $(1+\varepsilon_{n,1}) Z_{n,1} \alpha_{n,1}^2$ (note that the law of $Z_{n,1}$ depends on $\varepsilon$, but we omit this dependency in order to ease the notation). More concretely, this coupling can for instance be realized by adding a ``ghost'' individual in the population each time an individual in stage one makes an unsuccessful infection attempt (alternatively, we could also invoke the comparison result of Rogers and Williams~\cite[Theorem~V.$43.1$]{Rogers87:0}). Using the strong Markov property at time $T^\downarrow_\varepsilon(A_{n,1})$ and this coupling, we get
\[ \P \left( T_0(A_{n,1}) - T^\downarrow_\varepsilon(A_{n,1}) \geq \delta \right) \leq \P \left( T_0(Z_{n,1}) \geq \delta \right). \]

It is well-known that $Z_{n, 1} \Rightarrow Z_1$, where $Z_1$ is Feller diffusion with drift $\gamma_1$ started at $\varepsilon$ (see for instance Ethier and Kurtz~\cite[Chapter $9$]{Ethier86:0}). Moreover, standard arguments can be used to show that $T_0(Z_{n,1}) \Rightarrow T_0(Z_1)$, for instance by using the fact that $Z_{n,1}$ and $Z_1$ are time-change of L\'evy processes killed at $0$, say $Y_n$ and $Y$, so that $T_0(Z_{n,1}) = \int_0^\infty Y_n \Rightarrow \int_0^\infty Y = T_0(Z_1)$ (this time-change transformation is usually called Lamperti transformation, see for instance Lamperti~\cite{Lamperti67:1}). Thus we have (making clear the role of the initial condition)
\[ \limsup_{n \to +\infty} \P \left( T_0(A_{n,1}) - T^\downarrow_\varepsilon(A_{n,1}) \geq \delta \right) \leq \P \left( T_0(Z_1) \geq \delta \mid Z_1(0) = \varepsilon \right). \]

Since $T_0(Z_1) \Rightarrow 0$ as $Z_1(0) \to 0$, we have finally proved~\eqref{eq:limsup}, which concludes the first step.

\subsection{Second step}

The first step shows, by using the strong Markov property at time $T_0(A_{n,1})$, that we only need to prove Theorem~\ref{thm:outbreak} when $A_{n,1}(0) = 0$. In this case, Theorem~\ref{thm:SL-intermediate} shows that $\mv{A}_n \Rightarrow \mv{A}$, where $\mv{A}$ is a solution to the ODE~\eqref{eq:ODE} with $y = 0$.

With this initial condition, we have $A_{n,1}(t) = 0$ for all $t \geq 0$, and $\bar \pi \circ \mv{A}_n$ is a Markov process: actually, it is a multistage epidemic process with $K{-}1$ stages. The problem to iterate the arguments of the first step is that Proposition~\ref{prop:A} shows that, although $T^\downarrow_\varepsilon(A_2) < +\infty$ for every $\varepsilon > 0$, we have $T_0(A_2) = +\infty$. To get round this problem, we will use a time-change argument. Such an idea is classical in the SIR case $K = 1$, see for instance von Bahr and Martin-L\"of~\cite{Bahr80:0}.
\\

Let us set up a similar coupling as in the first step. The process $A_{n,2}$ decreases by $1/\alpha_{n,1}^2$ at rate $(1+\delta_{n,2}) \alpha_{n,1}^3 A_{n,2}$ and increases by $1/\alpha_{n,1}^2$ at rate
\[ (1+\varepsilon_{n,2}) \alpha_{n,1}^3 A_{n,2} (1-A_{n,0}/\alpha_{n,1}) \leq (1+\varepsilon_{n,2}) \alpha_{n,1}^3 A_{n,2} (1-A_{n,0}(0)/\alpha_{n,1}), \]
where the inequality follows from the monotonicity of $A_{n,0}$. Thus similarly as in the first step, we can couple $A_{n,2}$ with a continuous-time Markovian branching process $Z_{n,2}$ such that $Z_{n,2}(0) = A_{n,2}(0),$ $A_{n,2}(t) \leq Z_{n,2}(t)$ for $t \geq 0$, and $Z_{n,2}$ decreases by $1/\alpha_{n,1}^2$ at rate $(1+\delta_{n,2}) \alpha_{n,1}^3 Z_{n,2}$ and increases by $1/\alpha_{n,1}^2$ at rate $(1+\varepsilon_{n,2}) \alpha_{n,1}^3 Z_{n,2} (1-A_{n,0}(0)/\alpha_{n,1})$. In particular, $Z_{n,2} \Rightarrow z_2$ with $z_2(t) = z_2(0) \exp(-(A_0(0) - \gamma_2) t)$.

Since by Proposition~\ref{prop:A}, $A_0$ is strictly increasing with $A_0(\infty) > \gamma_2$, we can assume without loss of generality by shifting the processes at time $T^\uparrow_{\gamma_2}(A_{n,0}) + 1$ that $A_0(0) > \gamma_2$, so that each $z_2$ vanishes exponentially fast. The problem, as mentioned earlier, is that $T_0(z_2) = +\infty$: we now introduce the time-change argument.
\\

Let $C_{n,2}$ be the right-continuous inverse of $t \mapsto \int_0^t Z_{n,2}$ and $c_2$ be the right-continuous inverse of $t \mapsto \int_0^t z_2$, in the sense that $\int_0^{C_{n,2}(t)} Z_{n,2} = t$ for $t < \int_0^\infty Z_{n,2}$ and $\int_0^{c_2(t)} z_2 = t$ for $t < \int_0^\infty z_2$. Since $\int_0^\infty z_2 < +\infty$, $c_2$ blows up at time $\int_0^\infty z_2$. Moreover, such random time-change transformations induce continuous mappings, so that $\mv{A}_n \circ C_{n,2} \Rightarrow \mv{A} \circ c_2$ and $Z_{n,2} \circ C_{n,2} \Rightarrow z_2 \circ c_2$, see for instance Helland~\cite{Helland78:0}.

Time-changing $Z_{n,2}$ with $C_{n,2}$ actually corresponds to the Lamperti transformation mentioned above: $Z_{n,2} \circ C_{n,2}$ is a continuous-time random walk (killed at $0$), $z_2 \circ c_2$ starts at $A_2(0)$ and decays linearly at rate $A_0(0) - \gamma_2$, i.e., $z_2(c_2(t)) = A_2(0) - (A_0(0) - \gamma_2) t$ for $t \leq \int_0^\infty z_2$, and $T_0(Z_{n,2} \circ C_{n,2}) \Rightarrow T_0(z_2 \circ c_2)$. In particular, since
\[ \P \left( T_0(A_{n,2} \circ C_{n,2}) - T^\downarrow_\varepsilon(A_{n,2} \circ C_{n,2}) \geq \delta \right) \leq \P \left( T_0(Z_{n,2} \circ C_{n,2}) \geq \delta \mid Z_{n,2}(0) = \varepsilon \right) \]
we obtain
\[ \limsup_{n \to +\infty} \P \left( T_0(A_{n,2} \circ C_{n,2}) - T^\downarrow_\varepsilon(A_{n,2} \circ C_{n,2}) \geq \delta \right) \mathop{\longrightarrow}_{\varepsilon \to 0} 0 \]
and the arguments of the first step imply that
\[ \left( \mv{A}_n \circ C_{n,2}, T_0(A_{n,2} \circ C_{n,2}) \right) \Rightarrow \left( \mv{A} \circ c_2, T_0(A_2 \circ c_2) \right). \]

In particular, the strong Markov property at time $T_0(A_{n,2} \circ C_{n,2})$ shows that
\[ (\mv{A}_n \circ C_{n,2}) \left( T_0(A_{n,2} \circ C_{n,2}) \right) = \mv{A}_n(T_0(A_{n,2})) \Rightarrow (x \circ c_2) \left(T_0(A_2 \circ c_2) \right) = x(\infty) = 0.  \]

By using the strong Markov property at time $T_0(A_{n,2})$ and iterating this argument, we finally end up with the desired result that $A_{n,0}(\infty) \Rightarrow A_0(\infty)$, which concludes the proof of Theorem~\ref{thm:outbreak}.

\section{An intriguing conjecture} \label{sec:conjecture}

We conclude this paper by discussing a conjecture formulated in~\cite{Antal12:0} which initially motivated the present work.

\begin{conjecture*}[Antal and Krapivsky~\cite{Antal12:0}]
	Assume that $\varepsilon_{n,k} = \delta_{n,k} = 0$ and let $N_{n,k}$ be the number of individuals ever being of type $k = 1, \ldots, K$ over the course of the epidemic, starting from the initial condition $a_{n,1}(0) = 1$ and $a_{n,k}(0) = 0$ for $k = 2, \ldots, K+1$. Then $\E (N_{n,k})$ grows as $n \to +\infty$ like $n^{k \lambda_K}$, where
	\[ \lambda_K = \frac{2^K-1}{(K+1) 2^K - 1}. \]
\end{conjecture*}

Note that, with our notation, $N_{n,K} = a_{n,K+1}(\infty)$, since any individual ever removed must have been in stage $K$ of the epidemic at some point (and vice-versa). On the other hand, Theorem~\ref{thm:outbreak} shows that starting with of the order of $n^{1/(K+2)}$ individuals in stage one (instead of just one as in the above conjecture), $a_{n,K+1}(\infty)$ is, in distribution, of the order of $n^{(K+1)/(K+2)}$. When $K = 1$, there is a well-known argument that links these two objects: the connection goes through a random partitioning of $\{1,\ldots,n\}$.
\\

Consider indeed the following model, where $n$ individuals are assigned a unique label from the set $\{1,\ldots,n\}$ and which results in a random partition $\Pi_1, \ldots, \Pi_S$ of the set $\{1, \ldots, n\}$. Imagine that $\Pi_1, \ldots, \Pi_s$ have been generated and that the set $F_s = \{ 1, \ldots, n \} \setminus (\Pi_1 \cup \cdots \cup \Pi_s)$ is not empty: then the iteration proceeds as follows. Choose an individual $v$ uniformly at random from $F_s$, and run the epidemic with the following initial condition: at time $0$ the individuals in $F_s \setminus \{v\}$ are susceptible and the individuals in $\Pi_1 \cup \cdots \cup \Pi_s$ are removed, so that only $v$ is infected (and, in the case $K \geq 2$, is in the first stage of the epidemic). Eventually, this epidemic will die out and we define $\Pi_{s+1}$ as the set of individuals infected over the course of this epidemic.

This connection between random partition and epidemic processes is well-known, see for instance Barbour and Mollison~\cite{Barbour90:0}. In the case $K {=} 1$ this links clusters of the Erd\"os-R\'enyi random graph to the SIR process. This construction leads to several interesting by-products, one of them being that it makes it possible to compute the mean size of a typical cluster in terms of the mean size of the largest ones. More precisely, choose an individual $v$ uniformly at random in $\{1,\ldots,n\}$ and let $\Pi^*$ be the cluster to which it belongs. Let moreover $(\Pi_{(i)}, i \geq 1)$ be the clusters ordered in decreasing size, i.e., $\{\Pi_{(i)}\} = \{\Pi_i\}$ and $\lvert \Pi_{(1)} \rvert \geq \lvert \Pi_{(2)} \rvert \geq \cdots$ with $\lvert E \rvert$ the size of a set $E \subset \N$. Then the probability that $v$ belongs to $\Pi_{(i)}$ is exactly $\lvert \Pi_{(i)} \rvert/n$, in which case $\Pi^* = \Pi_{(i)}$ and so
\[ \E(\lvert \Pi^* \rvert) = \sum_{i \geq 1} \E \left( \frac{\lvert \Pi_{(i)} \rvert}{n} \lvert \Pi_{(i)} \rvert \right). \]

Assuming that the largest term in this sum dominates, we get the approximation
\begin{equation} \label{eq:approximation-susceptibility}
	\E\left(\lvert \Pi^* \rvert\right) \approx \frac{1}{n} \E\left(\lvert \Pi_{(1)}^2\rvert\right).
\end{equation}

When $K = 1$, $\lvert \Pi^* \rvert$ is, by exchangeability, equal in distribution to $N_{n,1}$ and so $\E(\lvert \Pi^* \rvert) = \E(N_{n,1})$. Moreover, $\lvert \Pi_{(1)} \rvert$ is of the same order than an outbreak size started from an \emph{intermediate} initial condition. Roughly speaking, this comes from the fact that the intermediate regime is precisely the one where the finite-size population effects begin to kick in. In particular, $\lvert \Pi_{(1)} \rvert$ is of the order of $n^{(K+1) / (K+2)} = n^{2/3}$. Gathering these two observations, we end up thanks to~\eqref{eq:approximation-susceptibility} with the relation $\E(N_{n,1}) \approx n^{4/3 - 1} = n^{1/3}$, when $K = 1$. This answer coincides with Antal and Krapivsky's conjecture and the above reasoning through random partitioning can be made rigorous, see, e.g., Janson and Luczak~\cite{Janson08:0}.
\\

It is tempting to also use this reasoning for $K \geq 2$, and this initially motivated us to identify the intermediate regime and the scalings at play there. However, this reasoning would lead to the estimate
\[ \E(N_{n,K}) \approx \frac{1}{n}n^{2(K+1)/(K+2)} = n^{K/(K+2)} \]
which is different (for $K \geq 2$) from the $n^{K \lambda_K}$ predicted by Antal and Krapivsky~\cite{Antal12:0}. We find this discrepancy, and the related fact that the above reasoning through random partitioning seems to fail, very intriguing. We believe that the temporal aspects intrinsic to this multistage epidemic (see the discussion of the main results in the introduction) play a major role in this discrepancy, although it is challenging to obtain rigorous results in that direction.

\appendix

\section{Proof of Lemmas~\ref{lemma:stab-ODE} and~\ref{lemma:properties-ODE}} \label{appendix:proof-ODE}

Let $x$ be any solution to~\eqref{eq:ODE} defined on the interval $J = [0,t^*)$ for some $t^* \in (0,\infty]$. It will be convenient to define $x_1 = y$ and to index the $\R^{K+1}$-valued function $x$ by the set $\{0,2,\ldots,K+1\}$, i.e., to write $x = (x_0, x_2, \ldots, x_{K+1})$. Let also in the sequel
\[ I_k(t) = \int_0^t \left(x_0(u) - \gamma_k\right) \d u, \ t \in J, k = 1, \ldots, K. \]

Then $(e^{I_k} x_k)' e^{-I_k} = (x_0 - \gamma_k) x_k + x'_k$ is equal to $x_{k-1}$ for $k = 2, \ldots, K$ by~\eqref{eq:ODE}. Thus for these $k$ we have $(e^{I_k} x_k)' e^{-I_k} = x_{k-1}$, which can be rewritten as
\begin{equation} \label{eq:integral-form}
	x_k(t) = \left( x_k(0) + \int_0^t x_{k-1}(u) e^{I_k(u)} \d u \right) e^{-I_k(t)}, \ k = 2, \ldots, K, \ t \in J.
\end{equation}

\subsection{Proof of Lemma~\ref{lemma:stab-ODE}}

Note first that the representation~\eqref{eq:integral-form} implies that $x(t) \in [0,\infty)^{K+1}$ for every $t \in J$. Indeed, since $x_1(t) = y(t) \geq 0$, this implies that $x_2(t) \geq 0$ and by induction on $k$, this implies that $x_k(t) \geq 0$ for every $k = 2, \ldots, K$ and $t \in J$. Since finally $x_0(0), x_{K+1}(0) \in [0,\infty)$ and $x'_0 = x'_{K+1} = x_K$ which has just been showed to stay non-negative, we obtain that $x_0$ and $x_{K+1}$ also stay non-negative.

Let us now prove Lemma~\ref{lemma:stab-ODE}, i.e., existence and uniqueness of solutions to~\eqref{eq:ODE}. Since $F$ is locally Lipschitz, the Picard-Lindelh\"of theorem implies local existence and uniqueness to~\eqref{eq:ODE}. To show this globally, we only have to show that local solutions do not explode. Since $x'_2 = y + (\gamma_2 - x_0) x_2$ and $x_2, x_0 \geq 0$, we obtain $x_2' \leq y + \gamma_2^+ x_2$ where $\gamma^+ = \max(0,\gamma)$ for any $\gamma \in \R$. Gronwall's lemma thus shows that $x_2$ does not explode. Similarly, for $k = 3, \ldots, K$ we have $x'_k \leq x_{k-1} + \gamma_k^+ x_k$ and so $x_k$  also does not explode. Finally, since $x_0' = x_{K+1}' = x_K$, solutions stay locally bounded which proves the global existence and uniqueness on $[0,\infty)$.

\subsection{Proof of Lemma~\ref{lemma:properties-ODE}}

We prove each property separately. As mentioned earlier, in the rest of the proof we define $x_1 = y$.
\\

\noindent \textit{Proof of~\ref{ODE:>0}.} Assume that $y(0) > 0$ or $a_2 > 0$: then it is clear from~\eqref{eq:integral-form} that $x_2(t) > 0$ for all $t > 0$. By induction, we see that $x_k(t) > 0$ for all $k = 2, \ldots, K$ and $t > 0$.
\\

\noindent \textit{Proof of~\ref{ODE:boundedness}.} Assume that $x_k(t) > 0$ for $k = 2, \ldots, K$ and $t > 0$, and that $x_0$ is bounded: we prove by backwards induction on $k$ that $\int_0^\infty x_k$ is finite for $k = 1, \ldots, K$ and that $x_0(\infty) > \gamma_k$ for $k = 2, \ldots, K$. For $k = K$, the fact that $\int_0^\infty x_K$ is finite comes from the fact that $x_0$ is bounded and non-decreasing, and so its derivative $x_K$ is integrable on $[0,\infty)$. Consider now any $2 \leq k \leq K$ and assume that $\int_0^\infty x_k$ is finite: we prove that $\int_0^\infty x_{k-1}$ is finite and that $x_0(\infty) > \gamma_k$.

Since $\int_0^\infty x_k$ is finite, there must exist a sequence $t_n \to +\infty$ such that $x_k(t_n) \to 0$. Moreover, we have by definition $x_k' = x_{k-1} + (\gamma_k - x_0) x_k$ and so integrating between times $0$ and $t_n$ we obtain
\[ \int_0^{t_n} x_{k-1} = x_k(t_n) - x_k(0) + \int_0^{t_n} (x_0(u) - \gamma_k) x_k(u) \d u \leq x_k(t_n) + (x_0(\infty) - \gamma_k) \int_0^{t_n} x_k. \]

Letting $n \to +\infty$, we obtain the inequality $\int_0^\infty x_{k-1} \leq (x_0(\infty) - \gamma_k) \int_0^\infty x_k$ which shows, by induction, that $\int_0^\infty x_{k-1}$ is finite and also that $x_0(\infty) > \gamma_k$ (since $\int_0^\infty x_{k-1} > 0$).
\\

\noindent \textit{Proof of~\ref{ODE:to0}.} Assume that $y(t) = 0$ for all $t \geq 0$ and that $a_2 > 0$: in particular, according to~\ref{ODE:>0}, we see that $x_k(t) > 0$ for every $t > 0$ and $k = 0, 2, \ldots, K$, a fact that will repeatedly be used in the sequel. We begin by proving the following formula:
\begin{equation} \label{eq:iteration}
	x_k(t) = \left( \sum_{i=0}^{k-2} x_{k-i}(0) \phi_{k,i}(t) \right) \times \exp \left( -\int_0^t (x_0(u) - \gamma_k) \d u \right), \ k = 2, \ldots, K, \ t \geq 0,
\end{equation}
where the functions $\phi_{k,i}$ for $k = 2, \ldots, K$ and $i = 0, \ldots, k-2$ are defined recursively by $\phi_{k,0}(t) = 1$ and for $i = 1, \ldots, k-2$,
\begin{equation} \label{eq:recursion-phi}
	\phi_{k,i}(t) = \int_0^t \phi_{k-1,i-1}(u) e^{\eta_k u} \d u, \ \text{ with } \ \eta_k = \gamma_{k-1} - \gamma_k \ \ (k = 3, \ldots, K).
\end{equation}


We prove~\eqref{eq:iteration} by induction on $k$. For $k = 2$, we have by~\eqref{eq:integral-form}, and since $x_1 = 0$, $x_2(t) = x_2(0) e^{-I_2(t)}$ which is precisely~\eqref{eq:iteration}. Assume now that~\eqref{eq:iteration} holds for $k \geq 2$ and let us prove it for $k + 1$: plugging in~\eqref{eq:iteration} into~\eqref{eq:integral-form}, we obtain
\[ x_{k+1}(t) = \left( x_{k+1}(0) + \int_0^t \left( \sum_{i=0}^{k-2} x_{k-i}(0) \phi_{k,i}(u) \right) e^{-I_k(u)} \times e^{I_{k+1}(u)} \d u \right) e^{-I_{k+1}(t)}. \]

Using $I_{k+1}(u) - I_k(u) = (\gamma_k-\gamma_{k+1})u = \eta_{k+1} u$, exchanging the integral and the sum and changing variables in the sum, we obtain
\[ x_{k+1}(t) = \left( x_{k+1}(0) + \sum_{i=1}^{k-1} x_{k+1-i}(0) \int_0^t \phi_{k,i-1}(u) e^{\eta_{k+1} u} \d u \right) e^{-I_{k+1}(t)} \]
from which we get~\eqref{eq:iteration} by~\eqref{eq:recursion-phi}. We now prove that $x_0(\infty)$ is finite by contradiction, so assume that $x_0(\infty) = +\infty$. Starting from the definition~\eqref{eq:recursion-phi} of the $\phi_{k,i}$'s, we get by induction that $\phi_{k,i}(t) \leq e^{i \eta^* t}$ with $\eta^* = 1 + \max_{3 \leq j \leq K} |\eta_j|$. In particular, we get from~\eqref{eq:iteration} for $k = K$ that
\[ x_K(t) \leq \left( \sum_{i=0}^{k-2} x_{k-i}(0) \right) e^{K \eta^* t} \times \exp \left( -\int_0^t (x_0(u) - \gamma_K) \d u \right). \]

Since $x_0(t) - \gamma_K > K \eta^*$ for $t$ large enough (since we are assuming $x_0(\infty) = +\infty$), the last display implies that $x_K$ converges to $0$ exponentially fast, and in particular $\int_0^\infty x_K < +\infty$. Since $x_0' = x_K$, $x_0(\infty)$ is finite, which yields the contradiction. Thus $x_0(\infty)$ must be finite and so the conclusions of~\ref{ODE:boundedness} apply, in particular $x_0(\infty) > \gamma_k$ for every $2 \leq k \leq K$.

We now complete the proof and show that $x_k(t) \to 0$ for every $k = 2, \ldots, K$: in view of~\eqref{eq:iteration} we only have to show that
\begin{equation} \label{eq:limit}
	\lim_{t \to +\infty} \left( \phi_{k,i}(t) e^{- \int_0^t (x_0 - \gamma_k)} \right) = 0
\end{equation}
for every $i = 0, \ldots, K-2$ and every $k = i+2, \ldots, K$. We prove this by induction on $i$: for $i = 0$ this comes immediately from the facts that $\phi_{k,0}(t) = 1$ and $x_0(\infty) > \gamma_k$. So assume that~\eqref{eq:limit} holds for some $i = 0, \ldots, K-2$ and every $k = i+2, \ldots, K$: we show that it also holds for $i+1$ and $k = i+3, \ldots, K$. By definition~\eqref{eq:recursion-phi} we have
\[ \phi_{k,i+1}(t) e^{- \int_0^t (x_0 - \gamma_k)} = \left(\int_0^t \phi_{k-1,i}(u) e^{\eta_k u} \d u\right) e^{- \int_0^t (x_0 - \gamma_k)}. \]

Let $\varepsilon > 0$ and, by induction hypothesis, $t^*$ such that $\phi_{k-1,i}(t) \leq \varepsilon e^{\int_0^t (x_0 - \gamma_{k-1})}$ for every $t \geq t^*$. Then for such $t$,
\begin{multline*}
	\phi_{k,i+1}(t) e^{- \int_0^t (x_0 - \gamma_k)} \leq \left(\int_0^{t^*} \phi_{k-1,i}(u) e^{\eta_k u} \d u\right) e^{- \int_0^t (x_0 - \gamma_k)}\\
	+ \varepsilon \left(\int_0^t e^{\int_0^u (x_0 - \gamma_{k-1})} e^{\eta_k u} \d u\right) e^{- \int_0^t (x_0 - \gamma_k)}.
\end{multline*}

Since the first term of the above upper bound vanishes as $t \to +\infty$ and we can rewrite the second term as
\[ \left(\int_0^t e^{\int_0^u (x_0 - \gamma_{k-1})} e^{\eta_k u} \d u\right) e^{- \int_0^t (x_0 - \gamma_k)} = \left(\int_0^t e^{\int_0^u (x_0 - \gamma_k)} \d u\right) e^{- \int_0^t (x_0 - \gamma_k)} = \int_0^t e^{-\int_u^t (x_0 - \gamma_k)} \d u, \]
we obtain
\[ \limsup_{t \to +\infty} \left( \phi_{k,i+1}(t) e^{- \int_0^t (x_0 - \gamma_k)} \right) \leq \varepsilon \sup_{t \geq 0} \left( \int_0^t e^{-\int_u^t (x_0 - \gamma_k)} \d u \right). \]

Thus to achieve the proof we only have to show that this last supremum is finite. Let $\kappa > 0$ and $s^* < +\infty$ be such that $x_0(t) - \gamma_k \geq \kappa$ for $t \geq s^*$ and $k = 2, \ldots, K$. Then for $t \geq s^*$
\[ \int_{s^*}^t e^{-\int_u^t (x_0 - \gamma_k)} \d u \leq \int_{s^*}^t e^{-\kappa (t-u)} \d u \leq \int_0^\infty e^{-\kappa u} \d u \]
and $\int_0^{s^*} e^{-\int_u^t (x_0 - \gamma_k)} \d u \leq \int_0^{s^*} e^{-\int_u^{s^*} (x_0 - \gamma_k)} \d u$ so that writing
\[ \int_0^t e^{-\int_u^t (x_0 - \gamma_k)} \d u = \int_0^{s^*} e^{-\int_u^t (x_0 - \gamma_k)} \d u + \int_{s^*}^t e^{-\int_u^t (x_0 - \gamma_k)} \d u \]
achieves the proof.

\bibliographystyle{plain}

\end{document}